\documentclass[11pt]{article}

\usepackage[pagebackref,colorlinks=true,pdfpagemode=none,urlcolor=blue,
linkcolor=blue,citecolor=blue]{hyperref}

\usepackage{amsmath,amsfonts,amssymb,amsthm}
\usepackage{mathrsfs}
\usepackage{bm}
\usepackage{color}
\usepackage{algorithm}
\usepackage{algorithmic}
\usepackage{graphicx}
\usepackage{subfigure}
\usepackage{wasysym}
\usepackage{verbatim}

\usepackage{common_style}
\usepackage{paper_style}

\begin{document}

%%%%%%%%%% TITLE %%%%%%%%%%
\title{Wavelet methods for shape perception in electro-sensing \thanks{\footnotesize This work was
    supported by ERC Advanced Grant Project MULTIMOD--267184.}}

\author{Habib Ammari\thanks{\footnotesize Department of Mathematics and Applications,
Ecole Normale Sup\'erieure, 45 Rue d'Ulm, 75005 Paris, France
(ammari@dma.ens.fr, wang@dma.ens.fr).}, St\'ephane
Mallat\thanks{\footnotesize Computer Science Department, Ecole
Normale Sup\'erieure, 45 Rue d'Ulm, 75005 Paris, France
(mallat@di.ens.fr, waldspur@di.ens.fr).},  Ir\`ene Waldspurger\footnotemark[3],  \and
Han Wang\footnotemark[2]}

%%%%%%%%%% MAIN BODY %%%%%%%%%%

\graphicspath{{./figures/}}

\maketitle
\bigskip

\begin{abstract}
  This paper aims at presenting a new approach to the electro-sensing problem using wavelets. It
  provides an efficient algorithm for recognizing the shape of a target from micro-electrical
  impedance measurements.  Stability and resolution capabilities of the proposed algorithm are
  quantified in numerical simulations.
\end{abstract}

\bigskip

\noindent {\footnotesize Mathematics Subject Classification
(MSC2000): 35R30, 35B30}


\section{Introduction}
\label{sec:intro}

The aim of electro-sensing is to learn geometric parameters and material compositions of a target
via electrical measurements. In this paper, we suppose that the target is composed of a homogeneous
material with a known electrical property and focus uniquely on the problem of geometry.  Geometric
identification of a target may mean to recognize it from a collection of known shapes (up to 
rigid transformations and scaling), or to reconstruct its boundary.

In the recent work \cite{ammari_target_2012}, an approach based on polynomial basis has been
proposed for the far-field measurement system. Using Taylor expansion of the Green functions, on one
hand, the geometric information of the target can be coded in some features, which are the action of
a boundary integral operator on homogeneous polynomials of different orders, and on the other hand
the measurement system is separated into a linear operator relating the features to the data. The
features are then extracted by solving a linear inverse problem and can be used to identify the
target in a database. Unlike other methods (\eg in electrical impedance tomography
\cite{borcea_electrical_2002}) which attempt to reconstruct directly the target, this approach is
more effective and computationally efficient in the applications of shape recognition.

From a more general point of view, the problem is to know, given the physical configuration of the
measurement system, how to choose the basis for representation of features and how to extract them
from data for identification. The ill-posedness in electro-sensing is inherent to the diffusion
character of the currents and cannot be removed by a change of basis. Nonetheless, the particularity of
a basis can modify totally the way in which information is organized in the feature and the manner in which it should be
reconstructed.

In this paper we present a new approach for electro-sensing with the near-field measurement system
using the wavelet basis. Unlike the far-field measurement configuration which is known to be exponentially
unstable, the near-field measurement system is much more stable and the data reside in a higher
dimensional subspace, hence one can expect to reconstruct more information of the target from the
data. With the near-field measurement system, the new approach based on wavelet presents more
advantages than the approach based on polynomials, for the reason that the wavelet representation of
the features is local and sparse, and reflects directly the geometric character of a
target. Furthermore, the features can be effectively reconstructed by seeking a sparse solution via
$\ell^1$ minimization, and the boundary of the target can be read off from the features, giving a
new high resolution imaging algorithm which is robust to noise.

This paper is organized as follows. In section \ref{sec:electric-sensing} we give a mathematical
formulation and present an abstract framework for electro-sensing. We introduce the basis of
representation and deduce a linear system by separating the features from the measurement system. The
question of the stability of the measurement system is discussed. In
section~\ref{sec:poly-basis-shape-dico} we summarize essential results based on polynomial basis
developed in \cite{ammari_target_2012}. The wavelet basis and new imaging algorithms, which are the
main contributions of this paper, are presented in section \ref{sec:WPT}, where we discuss some
important properties of the wavelet representation and formulate the $\ell^1$ minimization problem
for the reconstruction of the features. Numerical results are given in section \ref{sec:num_exp},
and followed by some discussions in section~\ref{sec:discussion}. The paper ends with some concluding remarks.

\section{Modelling of the electro-sensing problem}
\label{sec:electric-sensing}

Let $D\subset\R^2$ be an open bounded domain of ${\mcl C}^2$-boundary that we want to characterize
via electro-sensing. We suppose that $D$ is centered around the origin and has size $1$, furthermore
there exists $\Omega \subset [-1,1]^2$ an \emph{a priori} open bounded domain such that $D$ is
compactly contained in the convex envelope of $\Omega$ (in practice, both the center of $D$ and
$\Omega$ can be estimated using some \emph{location search algorithm}
\cite{ammari_mathematical_2013, seo2}).  We also assume that the positive conductivity number $\kcst
\neq 1$ of $D$ is known, and the background conductivity is $1$. We denote by $D^c = \R^2 \setminus
\overline{D}$.

A measurement system consists of $N_s$ sources $\set{x_s}_{s=1\ldots N_s}$, and $N_r$ receivers
$\set{y_r}_{r=1\ldots N_r}$ disposed on $\Omega$. The potential field $u_s$ generated by the point
source $x_s$ is the solution to the equation
\begin{equation}
  \label{eq:conductivity_problem}
  \left\{
    \begin{aligned}
      &\nabla . ((1+(\kcst-1)\chi_D) \nabla u_s) = \delta_{x_s} \ \text{ in } \R^2, \\
      & u_s - \Gammas  = O(\abs{x}^{-1})\ \text{ as } \abs{x}\rightarrow \infty,
    \end{aligned}
  \right.
\end{equation}
where $\chi_D$ is the indicator function of $D$, and
$\Gammas(x):=\Gamma(x-x_s)=\frac{1}{2\pi}\log\abs{x-x_s}$ is the background potential
field. Similarly, we denote $\Gammar(x):=\Gamma(x-y_r)$.

The difference $u_s-\Gammas$ is the perturbation of potential field due to the presence of $D$ in
the background,  and evaluated at the receiver $y_r$ it gives the measurement
\begin{align}
  \label{eq:response_matrix}
  V_{sr} = u_s(y_r) - \Gammas(y_r),
\end{align}
which builds the multistatic response matrix $\mV = (V_{sr})_{sr}$ by varying the source and
receiver pair. In this section, we show that with the help of a bilinear form, the problem can be
formulated through a linear system relating the data $\mV$ and the features of $D$.

\subsection{Layer potentials and representation of the solution}
\label{sec:repr_sol}
Recall the single layer potential $\Sgl{D}$:
\begin{align}
  \label{eq:SimpleLayer}
  \Sglf{D}{\phi}(x) = \int_{\p D} \Gamma(x-y)\phi(y)ds(y),
\end{align}
and the \emph{Neumann-Poincar\'e} operator $\Kstar{D}$:
\begin{align}
  \label{eq:Kstar}
  \Kstarf{D}{\phi}(x) = \frac{1}{2\pi} \int_{\p D}\frac{\seq{x-y,\nu_x}}{\abs{x-y}^2} \phi(y) ds(y),
\end{align}
where $\nu_x$ is the outward normal vector at $x\in\p D$. $\Kstar D$ is a compact operator on
$L^2(\p D)$ for a $\mcl C^2$ domain $D$ and has a discrete spectrum in the interval $(-1/2,
1/2]$. Therefore, the operator $(\lambda I - \Kstar D)$ is invertible on $L^2(\p D)$ for the
constant
\begin{align}
  \label{eq:lambda_def}
  \lambda = \frac{\kcst+1}{2(\kcst-1)}.
\end{align}
Moreover, its inverse $\lKstari D:L^2(\p D) \rightarrow L^2(\p D)$ is also bounded.  An important
relation is the jump formula:
\begin{align}
  \label{eq:jump_formula}
  \ddn{\Sglf{D}{\phi}}\Big|_{\pm} =
  \Paren{\pm \frac 1 2 I + \Kstar D}[\phi],
\end{align}
where $\p /\p\nu$ denotes the normal derivative across the boundary $\p D$ and $\pm$ indicate
the limits of a function from outside and inside of the boundary, respectively. Details on these
operators can be found in \cite{ammari_polarization_2007}. 

With the help of these operators, the solution $u_s$ of \eqref{eq:conductivity_problem} can be
represented as
\begin{align}
  \label{eq:us_repre}
  u_s(x) = \Gammas(x) + \Sglf{D}{\phi}(x)
\end{align}
with $\phi$ satisfying $\lKstarf{D}{\phi} = \p{\Gammas}/\p \nu$ on $\p D$. Therefore, the perturbed
field can be expressed as
\begin{align}
  \label{eq:repr_solution}
  (u_s-\Gammas)(x)=\int_{\p D} \Gamma(x-y) \lKstarif{D}{\ddn{\Gammas}}(y)ds(y).
\end{align}
We assume in the sequel $x_s, y_r\notin \p D$ for all $s,r$ which is necessary for $V_{sr}$ to be
well defined.

\subsection{Bilinear form $\TauD$} \label{sec:bilinear-form-taud}

We denote by $H^s(\Omega)$, for $s=1,2$, the standard Sobolev spaces and introduce the bilinear form
$\TauD:\HtO \times \HoO \rightarrow \R$ defined as follows
\begin{align}
  \label{eq:def_Tau_op}
  \TauD(f,g):=\int_{\p D} g(x)\lKstarif{D}{\ddn f} (x)\,ds(x) \ \text{ for } f\in \HtO, g\in \HoO.
\end{align}
By the boundeness of $\lKstari D$ and of the trace operator, it follows that
\begin{align*}
  \abs{\TauD(f,g)} \leq C \norm{f}_{\Ht} \norm{g}_{\Ho},
\end{align*}
and hence, $\TauD$ is bounded. Now, we observe from \eqref{eq:response_matrix} and
\eqref{eq:repr_solution} that $V_{sr}$ can be rewritten as
\begin{align}
  \label{eq:Vsr_TauD}
  V_{sr} = \TauD (\Gammas, \Gammar).
\end{align}

\subsubsection{Characterization of $D$ by $\TauD$}
\label{sec:char-d-taud}

One of the interests of $\TauD$ is that it determines uniquely the domain $D$, as stated in the
following result.
\begin{prop}
  \label{prop:charact-D}
  Let $D, D'\subset\Omega$ be open bounded domains with $\mcl C^2$-boundaries with the same conductivity $\kcst
  \neq 1$, then $D=D'$ if and only if their associated bilinear forms are equal:
  \begin{align}
    \label{eq:TauD_TauDp_eq}
    \TauDD D(f,g) = \TauDD{D'}(f,g) \ \forall f\in\HtO, g\in\HoO.
  \end{align}
\end{prop}

\begin{proof}
  Clearly $D=D'$ implies that $\TauDD D=\TauDD {D'}$. Now suppose $D\ne D'$. There exist a point
  $x\in\partial D\setminus\partial D'$. Let $\mathcal{V}$ be a small open neighborhood of $x$ such
  that $\mathcal{V}\cap\partial D'=\emptyset$.

  Let $f\in \mcl C_0^\infty(\Omega)$ verifying that the support of $f$ is included in $\mathcal{V}$
  and $\frac{\partial f}{\partial\nu}\not\equiv 0$ over $\mathcal{V}\cap\partial D$. Such a $f$ can
  be constructed for example in the following way. Let $\varphi$ be a compactly supported
  $\mathcal{C}^\infty$ function, whose support is included in $\mathcal{V}$ and $\phi=1$ in a small
  neighborhood of $x$. Then the function
  $$f:y\mapsto\varphi(y)\langle y,\nu_x\rangle$$ satisfies the required conditions: its support is
  included in $\mathcal{V}$ and $\frac{\partial f}{\partial\nu}(x)=1$ (in a neighborhood of $x$, $f$
  coincides with the function $y\mapsto\langle y,\nu_x\rangle$, whose gradient is $\nu_x$).

  We set $h=(\lambda I-\mathcal{K}^*_D)^{-1}\left[\frac{\partial f}{\partial \nu}\right]\in
  L^2(\partial D)$, which is not identically zero because $(\lambda
  I-\mathcal{K}_D^*)[h]=\frac{\partial f}{\partial \nu}$. Consequently, by the density of the image
  of the trace operator in $L^2(\p D)$, there exist $g\in \mcl C^\infty (\Omega)$ such
  that % $\int_{\partial D}g(x)h(x)ds(x)\ne 0$.  For this function $g$:
  %% it is equivalent to require g in H^1(\Omega). 
  \begin{equation*}
    \TauD (f,g)=\int_{\partial D}g(x)h(x)ds(x)\ne 0.
  \end{equation*}
  On the other hand, $(\lambda I-\mathcal{K}_{D'}^*)\left[\frac{\partial f}{\partial\nu}\right]=0$
  over $\partial D'$ because $\frac{\partial f}{\partial\nu}=0$ over $\partial D'$. So
  $\mathcal{T}_{D'}(f,g)=0$, which implies $\mathcal{T}_{D}\ne \mathcal{T}_{D'}$.
 \end{proof}

\subsection{Representation of $\TauD$}
\label{sec:repr-taud}
As suggested by Proposition \ref{prop:charact-D}, all information about $D$ is contained in
$\TauD$. This motivates us to represent $\TauD$ in a discrete form (features of $D$) that will
be estimated from the data.

\subsubsection{Basis of representation}
\label{sec:basis-h1}

Let $\Xbasis = \set{e_n\in L^2(\Omega)}_{n\in \N}$ be a Schauder basis of $L^2(\Omega)$. We denote
by $V_K$ the finite dimensional subspace spanned by $\set{e_n}_{n \leq K}$ and $P_K$ the orthogonal
projector onto $V_K$:
\begin{align}
  \label{eq:Proj_PK}
  P_K f = \inf_{g\in V_K} \norm{f-g}_{L^2(\Omega)}.
\end{align}
We require the following conditions on the basis $\Xbasis$:
\begin{itemize}
\item % $\Xbasis$ is total in $H^s(\Omega)$ for $s=1,2$, \ie,
  For any $f\in H^s(\Omega), s=1,2,$
  \begin{align}
    \label{eq:Basis_completeness}
    \norm{f - P_K f}_{H^s(\Omega)} \rightarrow 0 \ \text{ as } K\rightarrow +\infty.    
  \end{align}
\item There exists a function $u(s,t)$ such that for $s=1,2$ and some $t>s$, we have $u(s,t)>0$.
  Furthermore, it holds for any $f\in H^t(\Omega)$
  \begin{align}
    \label{eq:bernstein-poly-wvl}
    \norm{f - P_K f}_{H^s(\Omega)} \leq C K^{-u(s,t)} \norm{f}_{H^t(\Omega)} \ \text{ as }
    K\rightarrow +\infty
  \end{align}
  with the constant $C$ being independent of $K$ and $f$.
\end{itemize}
% Two examples of basis satisfying these conditions will be given in section
% \ref{sec:poly-basis-shape-dico} and section \ref{sec:WPT}.

\subsubsection{Polynomial basis}
\label{sec:polynomial-basis}
The first example of $\Xbasis$ is the homogeneous polynomial basis. The property
\eqref{eq:Basis_completeness} is a direct consequence of the following classical result (see
Appendix \ref{sec:proof-lemma-polybasis} for its proof):
\begin{lem}
  \label{lem:polynomial-basis-Hs}
  Let $\infabs\alpha:=\max_i \alpha_i$. The family of polynomials $\set{x^\alpha, \infabs\alpha \leq
    K}_{K\geq 0}$ is complete in $H^s(\Omega)$ for $s\geq 0$.
\end{lem}
Estimate \eqref{eq:bernstein-poly-wvl} can be obtained using an equivalent result of Legendre
polynomials established in \cite{canuto_approximation_1982}, with $u(s,t)=t-2s+1/2$.

\begin{comment}
  Actually, the subspace $V_K$ spanned by homogeneous polynomial basis is the same as that spanned
  orthogonal Legendre polynomials. Hence the projector $P_K$ \eqref{eq:Proj_PK} is given by the
  basis expansion under Legendre polynomials.
\end{comment}

\subsubsection{Wavelet basis}
\label{sec:comp-supp-wavel} 
Another example of $\Xbasis$ is the wavelet basis. Let $\tphi\in\mcl C_0^r(\R), r\geq 2$ be a
one-dimensional orthonormal scaling function generating a multi-resolution analysis
\cite{mallat_wavelet_2008}, and let $\tpsi\in L^2(\R)$ be a wavelet which is orthogonal to $\tphi$
and has $p>2$ zero moments. We construct the two-dimensional scaling function $\phi=\psi^0$ by
tensor product as $\phi(x_1,x_2)=\tphi(x_1)\tphi(x_2)$, and similarly we construct the wavelets
$\psi^{k}$ for $k=1,2,3$ by tensor product of $\tphi,\tpsi$ \cite{mallat_wavelet_2008}. We denote by
\begin{align*}
  \psi_{j,n}^{k}(x):=2^{-j}\psi^{k}(2^{-j}x-n), \ j\in\Z, n\in\Z^2.
\end{align*}
Then $\set{\psikjn}_{k,j,n}$ for $j\in\Z, n\in\Z^2, k=1,2,3,$ constitute an orthonormal basis of
$L^2(\R^2)$. Particularly, the Daubechies wavelet of order 6 (with 6 zero moments) fulfills the
conditions above \cite{cohen_new_1996}.

Let $V_j$ be the approximation space spanned by $\set{\phi_{j,n}}_{n\in\Z^2}$, and $P_j$ be the
orthogonal projector onto $V_j$:
\begin{align}
  \label{eq:wvl_Pj_dfn}
  P_j f = \sum_{n\in\Z^2} \seq{f,\phi_{j,n}}\phi_{j,n}.
\end{align}
The property \eqref{eq:Basis_completeness} follows from the fact that the $P_j f$ converges to $f$
in $H^s(\Omega)$ for any $\abs s\leq r$ (see \cite[Theorem 6, Chapter 2]{meyer_wavelets_1992}).

The wavelet basis introduced above verifies the polynomial exactness of order $p-1$
\cite{mallat_wavelet_2008} (\ie, the polynomials of order $p-1$ belong to $V_0$) and $\phi \in
H^s(\R^2)$ for $s=1,2$. Therefore, we have the following result (see \cite[Corollary
3.4.1]{cohen_numerical_2003}): For any $f\in H^t(\Omega), s<t \leq p$
\begin{align}
  \label{eq:inverse_estime_Hst}
  \norm{P_j f-f}_{H^s(\Omega)} \lesssim 2^{j(t-s)} \norm{f}_{H^{t}(\Omega)} \ \text{ as } j
  \rightarrow -\infty.
\end{align}
Then the estimate \eqref{eq:bernstein-poly-wvl} is fulfilled with $u(s,t)=(t-s)/2$. By abuse of
notation, throughout this paper, we still use $P_K$ (with $K\propto 2^{-2j}$) to denote the
projection for the wavelet basis.

\subsubsection{Truncation of $\TauD$}
\label{sec:truncation-taud} 

Thanks to the boundedness of $\Tau$ and property \eqref{eq:Basis_completeness}, one can verify
easily that for any $f\in \HtO, g\in \HoO$
\begin{align}
  \label{eq:TauD_fg_TauD_fkgk}
  \TauD(f,g) = \Tau(P_K f, P_K g) + o(1),
\end{align}
with the truncation error $o(1)$ decaying to zero as $K\rightarrow +\infty$.  Using the
approximation property \eqref{eq:bernstein-poly-wvl}, a bound on the truncation error $o(1)$ can be
established.
\begin{prop}
  \label{prop:WPT_Tau_apprx}
  Suppose that the basis $\Xbasis$ fulfills the conditions \eqref{eq:Basis_completeness} and
  \eqref{eq:bernstein-poly-wvl}. Let $\tilde u(t,t'):=\min(u(2,t), u(1,t'))$ with the constants
  $t>2$ and $t'>1$ being those of estimate \eqref{eq:bernstein-poly-wvl}. Then for any $f\in
  H^t(\Omega), g\in H^{t'}(\Omega)$
  \begin{align}
    \label{eq:Tau_apprx_Tau_error}
    \Abs{\TauD(P_K f,P_K g) - \TauD(f,g)} \leq C K^{-\tilde u(t,t')} \ \text{ as } K\rightarrow +\infty,
  \end{align}
  where the constant $C$ depends only on $f,g, t,$ and  $t'$.
\end{prop}
\begin{proof}
  By the triangle inequality
  \begin{align}
    \label{eq:Trunc_err_TauD_triangl}
    \Abs{\TauD(P_Kf, P_Kg) - \TauD(f, g)} &\leq \Abs{\TauD(P_Kf-f, P_Kg)} + \Abs{\TauD(f, P_Kg -g)}.
  \end{align}
  Using the  boundedness of $\TauD$ on the first term of the right-hand side, we get
  \begin{align*}
    \Abs{\TauD(P_Kf-f, P_Kg)} \leq C \norm{P_K f-f}_{H^2} \norm{P_K g}_{H^1}.
  \end{align*}
  On one hand, one can apply \eqref{eq:bernstein-poly-wvl} on $\norm{P_K f-f}_{H^2}$ with the
  constant $t>2$ verifying $u(2,t)>0$. On the other hand, for any $t'>1$ and $g\in H^{t'}(\Omega)$
  we have $P_K g\rightarrow g$ in $H^1(\Omega)$ by \eqref{eq:Basis_completeness}. Therefore, we
  obtain that
  \begin{align*}
    \Abs{\TauD(P_Kf-f, P_Kg)} \lesssim K^{-u(2,t)} \norm{f}_{H^t} \norm{g}_{H^1}.
  \end{align*}
  Similarly, for the second term of the right-hand side in~\eqref{eq:Trunc_err_TauD_triangl} we get
  \begin{align*}
    \Abs{\TauD(f, P_Kg -g)} \leq C \norm{f}_{H^2} \norm{P_Kg-g}_{H^1} \lesssim K^{-u(1,t')}
    \norm{f}_{H^2}\norm{g}_{H^{t'}},
  \end{align*}
  which holds for any $g\in H^{t'}(\Omega)$ for the constant $t'>1$ of
  \eqref{eq:bernstein-poly-wvl}. Combining these two terms yields the desired result.
\end{proof}

\subsubsection{Coefficient matrix}
\label{sec:coefficient-matrix}
We define the coefficient matrix $\mX$ as follows
\begin{align}
  \label{eq:Xcoeff_mat}
  \mX = \mX[D,K] = {(\TauD(e_m, e_n))}_{mn} \ \text{ for } m,n=1\ldots K,
\end{align}
which represents $\TauD$ under the basis $\Xbasis$ up to the order $K$. We denote by $\mbf f$ the
coefficient vector of $P_\vidx f$, \ie,
\begin{align}
  \label{eq:vecf_PKf}
  P_\vidx f = \sum_{n\leq \vidx} {\mbf f}_n e_n,
\end{align}
and similarly $\mbf g$ for $P_\vidx g$. Then $\TauD$ restricted on $V_K$ can be put into the following matrix
form:
\begin{align}
  \label{eq:TauD_K_repr}
  \TauD(P_K f, P_K g) = \sum_{m,n=1}^K {\mbf f}_m\TauD(e_m, e_n) {\mbf g}_n = {\mbf f}^\top \mX
  {\mbf g}.
\end{align}
Thanks to property \eqref{eq:Basis_completeness} of the basis, there is a one-to-one mapping
between $\TauD$ and its coefficient matrix $\mX$ as $K\rightarrow +\infty$. Hence the domain $D$ is
also uniquely determined from $\mX[D]$ when $K\rightarrow +\infty$, as a consequence of Proposition
\ref{prop:charact-D}.

\subsection{Linear system}
\label{sec:linear-system-VLX}
% if $\Omega$ contains any source $x_s$ or receiver $y_r$, the functions $\Gammas$ and $\Gammar$
% restricted on $\Omega$
%  (singularity of),

In~\eqref{eq:Vsr_TauD}, the Green functions $\Gammas$ and $\Gammar$ play the role of the measurement
system while the information about $D$ is contained in the operator $\TauD$. This motivates us to
separate these two parts and extract information about $D$ from the data $\mbf V$.
%which constitutes the first step in our approach to the electro-sensing problem.

By removing a small neighborhood of $x_s$ and $y_r$ if necessary, we can always assume that $\Omega$ does
not contain any source or receiver, in such a way that $\Gammas$ and $\Gammar$ restricted on
$\Omega$ become $\mcl C^\infty$, and hence can be represented using the basis $\Xbasis$ (note that since
$x_s,y_r\notin \p D$, removing the singularity does not affect $\TauD(\Gammas, \Gammar)$ which
depends only on the value of the  Green functions on $\p D$).

We denote in the sequel $\bgamma_{x_s}, \bgamma_{y_r}\in\R^K$ the (column) coefficient vectors of
$P_K \Gammas$ and $P_K\Gammar$, respectively. From \eqref{eq:Vsr_TauD}, \eqref{eq:TauD_fg_TauD_fkgk}
and \eqref{eq:TauD_K_repr} one can write
\begin{align}
  \label{eq:Vsr_bgamma_E}
  V_{sr} = \TauD (\Gammas, \Gammar) = \bgammas^\top \mX \bgammar + E_{sr}
\end{align}
with $E_{sr}$ being the truncation error of order $K$ which can be controlled using Proposition
\ref{prop:WPT_Tau_apprx}. We introduce the matrices of the measurement system
\begin{align}
  \label{eq:Gamma_x_y_mat}
  \mGx = [\bgamma_{x_1} \ldots \bgamma_{x_{N_s}}], \ \ \mGy = [\bgamma_{y_1} \ldots \bgamma_{y_{N_r}}],
\end{align}
as well as  the linear operator $\mapop{\mL}{\R^{K\times K}}{\R^{N_s \times N_r}}$:
\begin{align}
  \label{eq:def_linop_L}
  \mL(\mX) = {\mGx}^\top \mX \mGy .
\end{align}
Then \eqref{eq:Vsr_bgamma_E} can be put into a matrix product form:
\begin{align}
  \label{eq:MSR_TauDX_linsys0}
  \V = \mGx^\top \mX \mGy + \mE = \mL(\mX) + \mE
\end{align}
with $\mbf E=(E_{sr})_{sr}$ being the matrix of truncation error. Further, suppose that $\V$ is
contaminated by some measurement noise $\Wnoise$, \ie, the $m,n$-th coefficient follows the normal
distribution
\begin{align*}
  (\Wnoise)_{mn} \iid \normallaw {\snoise}
\end{align*}
with $\snoise>0$ being the noise level. Using the bound \eqref{eq:Tau_apprx_Tau_error} of the
truncation error,  we can assume that for a large
$K$ %the truncation error is much smaller than the measurement noise:
\begin{align}
  \label{eq:regime_Esr_noise}
  \abs{E_{sr}} \ll \snoise
\end{align}
uniformly in all $s$ and $r$, so that $\mE$ can be neglected compared to the noise. Finally, we
obtain a linear system relating the coefficient matrix to the data
\begin{align}
  \label{eq:MSR_TauDX_linsys}
  \V = \mL(\mX) + \Wnoise,
\end{align}
and the objective is then  to estimate $\mX$ from $\V$ by solving \eqref{eq:MSR_TauDX_linsys}.

% Using the bound \eqref{eq:Tau_apprx_Tau_error} on the truncation error, we assume that for a large
% $K$ the truncation error $\abs{E_{sr}} \ll \snoise$ uniformly in all $s$ and $r$, so that $\mE$ can
% be neglected in the linear system \eqref{eq:MSR_TauDX_linsys}. On the other hand, the maximal
% resolving order $K$ \cite{ammari_pams} that one can achieve by reconstruction depends clearly on
% $\snoise$ and the stability of the operator $\mL$. 

% However the estimation of the maximal resolving order $K$ \cite{ammari_pams} which depends mainly on
% $\snoise$ and $\mL$, is more subtle. We shall not present general results (independent of the choice
% of $\Xbasis$) in this direction, but only illustrate numerically the situation in the following
% section.

% and $\mX$ is reconstructed by solving~\eqref{eq:MSR_TauDX_linsys} via a linear or
% nonlinear estimator \cite{mallat_wavelet_2008}, in function of the basis $\Xbasis$ and the
% measurement system.

% Remark that in this expression, all information (up to the order $K$) about the shape $D$ is
% contained in $\mX$, while the measurement system is separated into the linear operator $\mL$.
% From \eqref{eq:MSR_TauDX_linsys} raise the questions of how to reconstruct $\mX$ and use it for
% shape perception and classification. The answers clearly depend on the choice of $\Xbasis$, which
% determines both $\mX$ and $\mL$.

\subsubsection{Measurement systems and stability}
\label{sec:stab-inter-meas}
The stability of the operator $\mL$ is inherent to the spatial distribution of sources and receivers
that we suppose to be coincident in what follows. The far-field measurement system
(Figure~\ref{fig:acq_systems} (a)) is the situation when the characteristic distance $\rho$ between
transmitters and the boundary of the target is much larger than the size $\delta$ of the target. On
the contrary, in the near-field internal measurement system (Figure~\ref{fig:acq_systems} (b)) which
is used in micro-EIT \cite{seo, pnas}, we have $\rho\ll \delta$ and the transmitters can be placed
``inside'' the target. Other types of far-field measurements exist;  see
\cite{ammari_modeling_2012,ammari_shape_2013}.

\def\figwidth{5.5cm}
\begin{figure}[htp]
  \centering
  \subfigure[Far-field measurement]{\includegraphics[width=\figwidth]{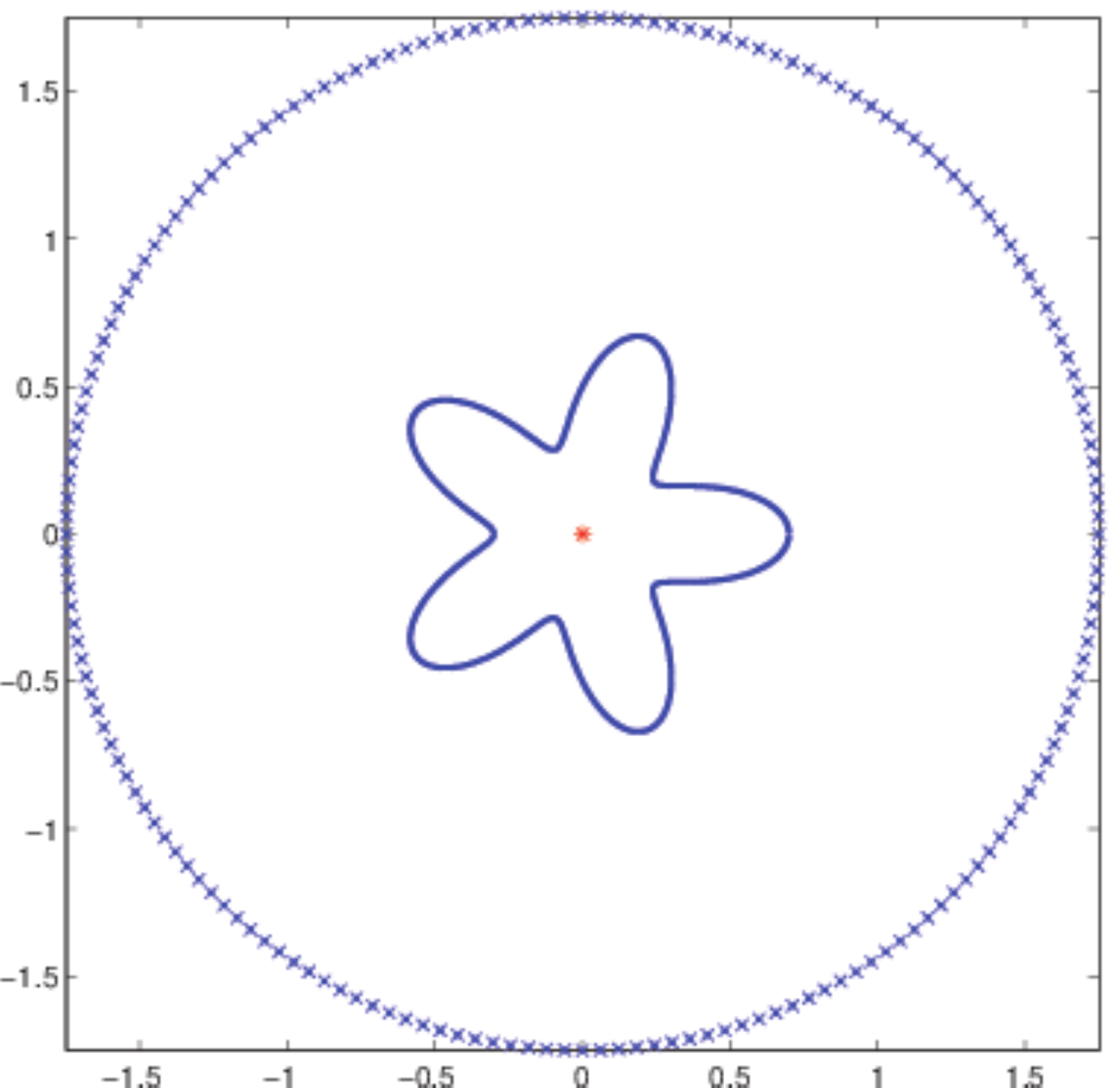}}
  \subfigure[Near-field measurement]{\includegraphics[width=\figwidth]{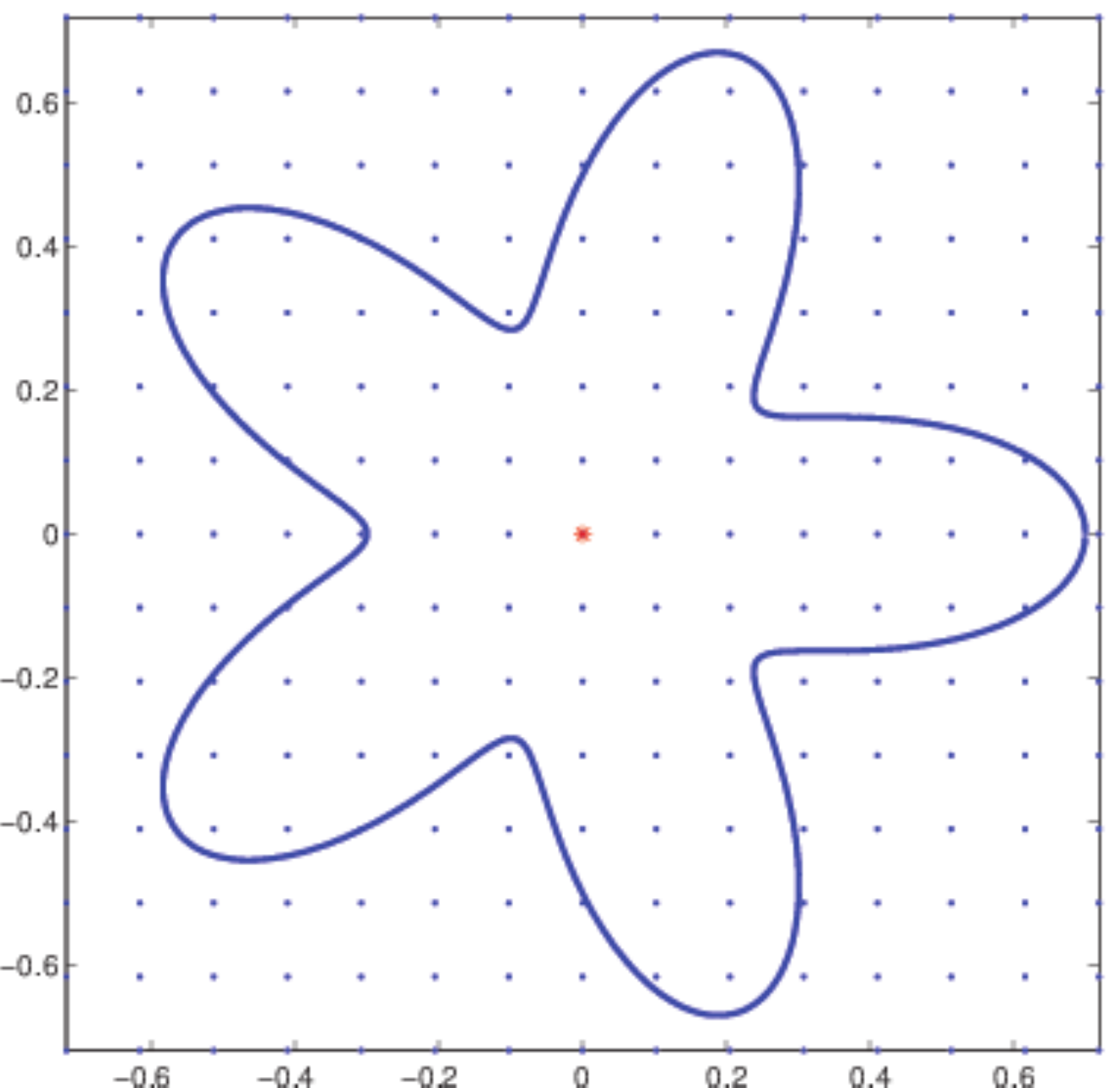}}
  \caption{Far-field and near-field measurement systems with a flower-shaped target $D$. (a): 360
    coincided sources and receivers (cross) are uniformly distributed on a circle including the
    target. (b): 15$\times$ 15 coincided sources and receivers (dots) are uniformly distributed
    inside $\Omega$, and the minimal distance to $\p D$ is $\sim 10^{-3}$. The red cross in both
    figures marks the center of $\Omega$.}
  \label{fig:acq_systems}
\end{figure}

Note that a singular value of $\mL$ is the product of a pair of singular values of $\mGx$ and
$\mGy$. In Figure~\ref{fig:svd_acq_systems} we compare the distribution of the singular values of the
operator $\mL$ (computed with the Daubechies wavelet of order 6 as $\Xbasis$) corresponding to the
systems of Figure~\ref{fig:acq_systems}. One can notice the substantial difference between these two
systems: the singular values of the far-field system decays very fast, revealing the exponential
ill-posedness of the associated inverse problem \cite{ammari_target_2012,ammari_tracking_2012,
  ammari_generalized_2011}. On the contrary, for the near-field system,  the stability of $\mL$ is
considerably improved.  This can be explained by the decay of the wavelet coefficients of the Green
function. In fact, $\Gammas$ is smooth away from $x_s$, therefore most rows (populated by detail
wavelet coefficients) in $\mGx$ have tiny numerical values which make the matrix
ill-conditioned. More precisely, the following result holds. We refer to   Appendix
\ref{sec:proof-decay-wvl-coeff}) for its proof. 
\begin{prop}
  \label{prop:decay_wavelet_coeff_Gammas}
  Let $\COmega$ be a compact domain. Suppose that $x_s\notin\COmega$ and denote by $\rho$ the
  distance between $x_s$ and $\COmega$. If the wavelet $\psi^k, k=1,2,3$ has $p>0$ zero moments, then
  as $j\rightarrow -\infty$:
  \begin{align}
    \label{eq:innerprod_gammas_psi_one_term}
    \abs{\seq{\Gammas, \psikjn}} \asymp 2^{j(p+1)} \rho^{-p} \text{ for } n\in \Lambda_j^k,
  \end{align}
  where $\Lambda_j^k:=\set{n\in\Z^2\, |\, \text{ the support of } \psikjn \text{ intersects with }
    \COmega}$.
\end{prop}

\begin{figure}[htp]
  \centering
  {\includegraphics[width=8cm]{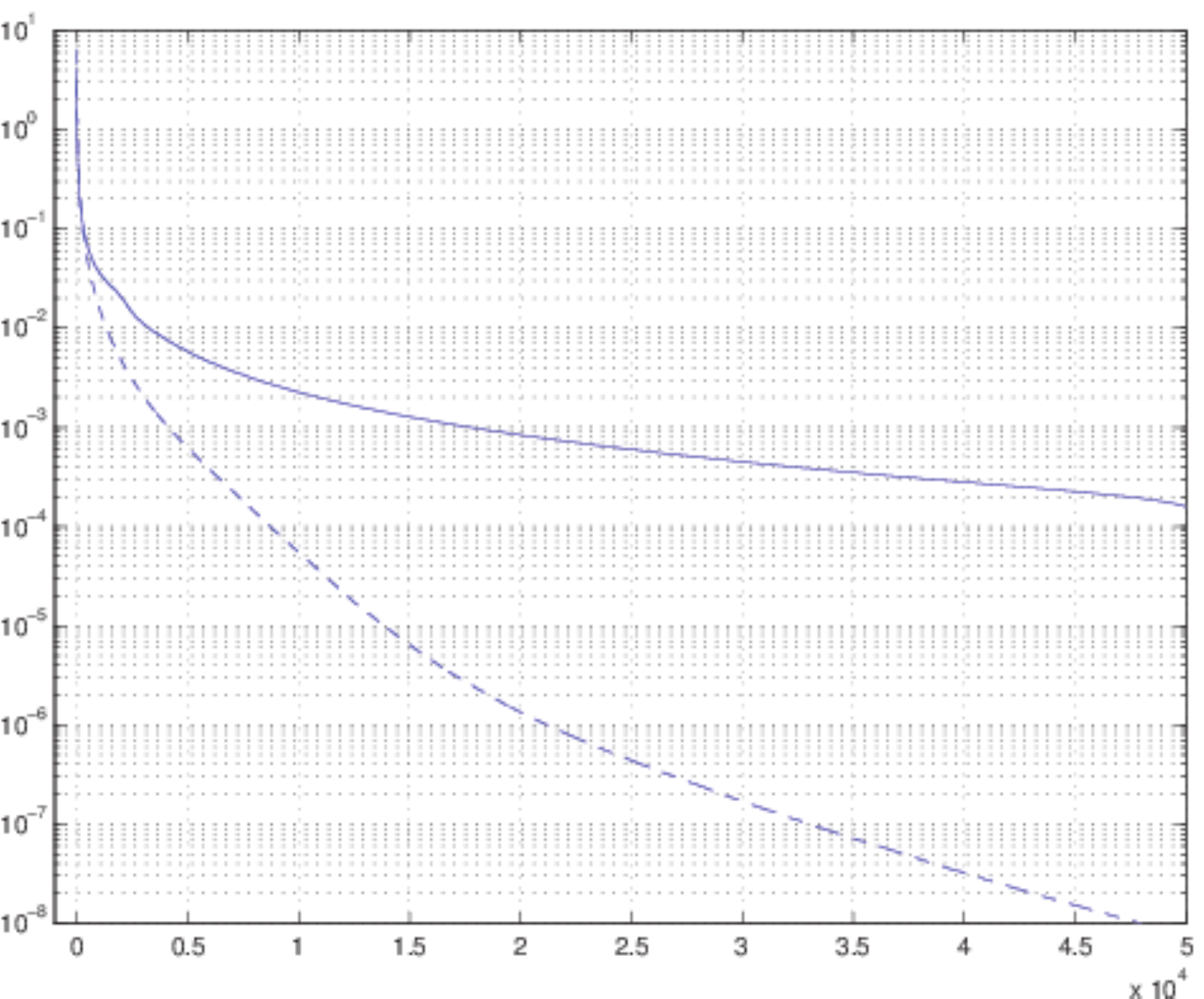}}
  \caption{First $5\times 10^4$ singular values of $\mL$ computed with a wavelet basis. The curve with
    dashed and solid line correspond respectively to the far-field and the near-field measurement
    systems of Figure~\ref{fig:acq_systems}.}
  \label{fig:svd_acq_systems}
\end{figure}

As a consequence of the stability, the data $\V$ of the far-field measurement reside in a low
dimensional subspace while those of the near-field measurement reside in a high dimensional
subspace. Therefore, the type of estimation of $\mX$ and \emph{afortiori} the type of basis for the
implementation of estimator, should be adapted to the physical configuration of the measurement
system. In the next sections, we show that in the case of the far-field measurement, the polynomial
basis with linear estimation is well suited, while for the near-field measurement case it is possible to
use a wavelet basis which creates a high dimensional but sparse matrix $\mX$, and to reconstruct more
information by a nonlinear estimator.

\section{Polynomial basis and linear estimation} %shape identification in a dictionary}
\label{sec:poly-basis-shape-dico}

Under the homogeneous polynomial basis, the coefficients in $\mX$ are given as follows:
\begin{equation}
  \label{eq:GPT_def}
  \mX_{\alpha \beta} := \TauD(x^\alpha, x^\beta)
\end{equation}
with $\alpha,\beta$ being multi-indices. These $\mX_{\alpha \beta}$ are also referred as
\emph{Generalized Polarization Tensors} (GPTs) in the perturbation theory of small inclusions
\cite{ammari_mathematical_2013, ammari_polarization_2007}. The coefficient vectors $\bgamma_{x_s}$
and $\bgamma_{y_r}$ in \eqref{eq:Vsr_bgamma_E} are now obtained by the Taylor expansion of the Green
functions:
\begin{align}
  \label{eq:GPT_expansion}
  V_{sr} = \sum_{\substack{\abs\alpha, \abs\beta=1}}^K
  \frac{(-1)^{\abs\alpha+\abs\beta}}{\alpha!\beta!} \p^\alpha \Gamma(x_s) \mX_{\alpha\beta}[D]
  \p^\beta\Gamma(y_r) + E_{sr}.
\end{align}
Moreover, the truncation error $E_{sr}$ can be expressed explicitly, and in case of the far-field
measurement it decays as $O((\delta/\rho)^{K+2})$ \cite{ammari_target_2012} (with $\delta$ being the
size of $D$ and $\rho>\delta$ being the radius of the measurement circle), which is far better than
the previous bound \eqref{eq:Tau_apprx_Tau_error}. 

Expression \eqref{eq:GPT_expansion} can be simplified to the matrix product form
\eqref{eq:MSR_TauDX_linsys0} by recombining all $\mX_{\alpha\beta}$ of order $\abs\alpha=m$ and
$\abs\beta=n$ using coefficients of harmonic polynomials. The resulting linear operator $\mbf L$ is
injective for the far-field system (Figure~\ref{fig:acq_systems} (a)) having $N_S>2K$ transmitters,
and its singular value decays as $\lambda_{mn}=O((mn)^{-1}(\delta/\rho)^{\frac{m+n}{2}})$; see
\cite{ammari_target_2012, ammari_tracking_2012}.

\subsection{Linear estimator of $\mX$}
\label{sec:linear-estimator-mx}

Due to the global character of the polynomials, in general the coefficient matrix $\mX$ is full. On
the other hand, the fast decay of truncation error under the polynomial basis suggests that the
energy of $\mX$ is concentrated in the low order coefficients. Therefore, the simple truncation in
the reconstruction order provides an effective regularization, and the first $K$ order coefficients
can be estimated by solving the least-squares problem
\begin{align}
  \label{eq:lin_estime_CGPT}
  \Xest := \argmin_{\mX} \Norm{\mbf{L(\mX)-V}}^2_F,
\end{align}
where $\norm{\cdot}_F$ denotes the Frobenius norm. The following bound on the error of
the estimation can be established: for $m,n=1\ldots 2K$, 
\begin{align}
  \label{eq:var_lin_estime_CGPT}
  \sqrt{\Exp{\Paren{(\Xest)_{mn}-(\mX)_{mn}}^2}} \leq C \frac{\snoise}{N_S} mn
  {(\delta/\rho)}^{-\frac{m+n}{2}}.
\end{align}
As a consequence, the maximum resolving order $K$ is bounded by
\cite{ammari_target_2012}
\begin{align}
  \label{eq:max_K_linear}
  K\lesssim \log_{\delta/\rho}\snoise.
\end{align}
Hence, the far-field measurement has a very limited resolution. However, the first few orders of
coefficients contain important geometric information of the shape (\eg the first order tells how a
target resembles an equivalent ellipse), and can be used to construct shape descriptors for the
identification of shapes. We refer the reader to \cite{ammari_target_2012} for detailed numerical
results.

\section{Wavelet basis and nonlinear estimation} %shape perception}
\label{sec:WPT}

In this section we use the wavelet (more exactly, the scaling function $\phi$) introduced in section
\ref{sec:basis-h1} for the representation $\TauD$ and the reconstruction of $\mX$. This yields a
sparse and local representation, and makes the wavelets an appropriate choice of basis for the
near-field measurement which allows the reconstruction via $\ell^1$ minimization and the visual
perception of $\p D$.

\subsection{Wavelet coefficient matrix}
\label{sec:wavel-coeff-matr}

For the wavelet basis, we use the scale number $L$ (in place of $K$ as in section
\ref{sec:repr-taud}) to denote the truncation order. The coefficient matrix $\mX=\mX[D,L,L]$ under
the wavelet basis contains the approximation coefficients
\begin{align}
  \label{eq:A_mat}
  \mX_{n,n'} &= \TauD(\phi_{L,n},\phi_{L,n'}),
\end{align}
so that $\TauD(P_L f,P_L g) = \mbf f^\top \mX \mbf g$ with $P_L$ being the orthogonal projector onto the
approximation space $V_L$ as introduced in section \ref{sec:comp-supp-wavel}, and $\mbf f, \mbf g$
are the coefficient vectors
\begin{align}
  \label{eq:Coeff_vec_std}
  \mbf f = [\set{\seq{f,\phi_{L,n}}}_n],\  \mbf g = [\set{\seq{g,\phi_{L,n}}}_n].
\end{align}
The matrix $\mX$ of a flower-shaped target is shown in Figure~\ref{fig:diag_local} (a).

\subsection{Properties of the wavelet coefficient matrix}
\label{sec:properties-mx}
In the following we discuss some important properties of the wavelet coefficient matrix.

\paragraph{Bounds of matrix norm}
\label{sec:bounds-normmx}
Proposition \ref{prop:bounds-mx} establishes bounds on the spectral norm of $\mX$ showing that
$\norm{\mX}_2$ diverges as $L\rightarrow -\infty$. The proof is based on the inverse estimate and
the polynomial exactness of the wavelet basis, and is given in
Appendix~\ref{sec:proof-proposition-bound-X}.
\begin{prop}
  \label{prop:bounds-mx}
  When $L\rightarrow -\infty$, the spectral norm of the matrix $\mX$ is bounded by
  \begin{align}
    \label{eq:bound_WPT_mat_l2}
    C' 2^{-L} \leq \norm{\mX}_2 &\leq C 2^{-3L}
  \end{align}
  with $C,C'>0$ being some constants independent of $L$.
\end{prop}

\paragraph{Sparsity}
\label{sec:sparsity-mx}
From the definition of $\TauD$ we observe that $\TauD(\phijn,\phijnp)$ is non-zero only when the
support of both wavelets intersect $\p D$. Therefore, the non-zero coefficients of $\mX$ carry
geometric information on $\p D$. Moreover, $\mX$ is a sparse matrix. In fact, when the scale
$L\rightarrow -\infty$ there are $\sim 2^{-2L}$ wavelets contributing to $D$, so the dimension of
$X$ is $\sim 2^{-4L}$. On the other hand, the number of wavelets intersecting $\p D$ is $\sim
2^{-L}$. Hence, the number of non zero coefficients is about $2^{-2L}$ and the sparsity of $\mX$ is
asymptotically $2^{2L}$.

\paragraph{Band diagonal structure}
\label{sec:band-diag-struct}
Numerical computations show that the pattern of non-zeros in $\mX$ has a band diagonal structure.
The largest coefficients appear around several principal band diagonals in a regular manner that
reflects different situations of interaction between wavelets via the bilinear form $\TauD$, as
shown in Figure~\ref{fig:diag_local} (a). We notice that the major band diagonals describe the
interaction between a $\phi_{L,n}$ and its immediate neighbors (the width of the band is
proportional to the size of the support of $\phi_{L,n}$). In particular, the main diagonal
corresponds to the case $n=n'$, while the other band diagonals describe the interactions between
other non-overlapping $\phiLn$ and $\phiLnp$.

\begin{figure}[htp]
  \centering
  \subfigure[$\mX$]{\includegraphics[width=7.5cm]{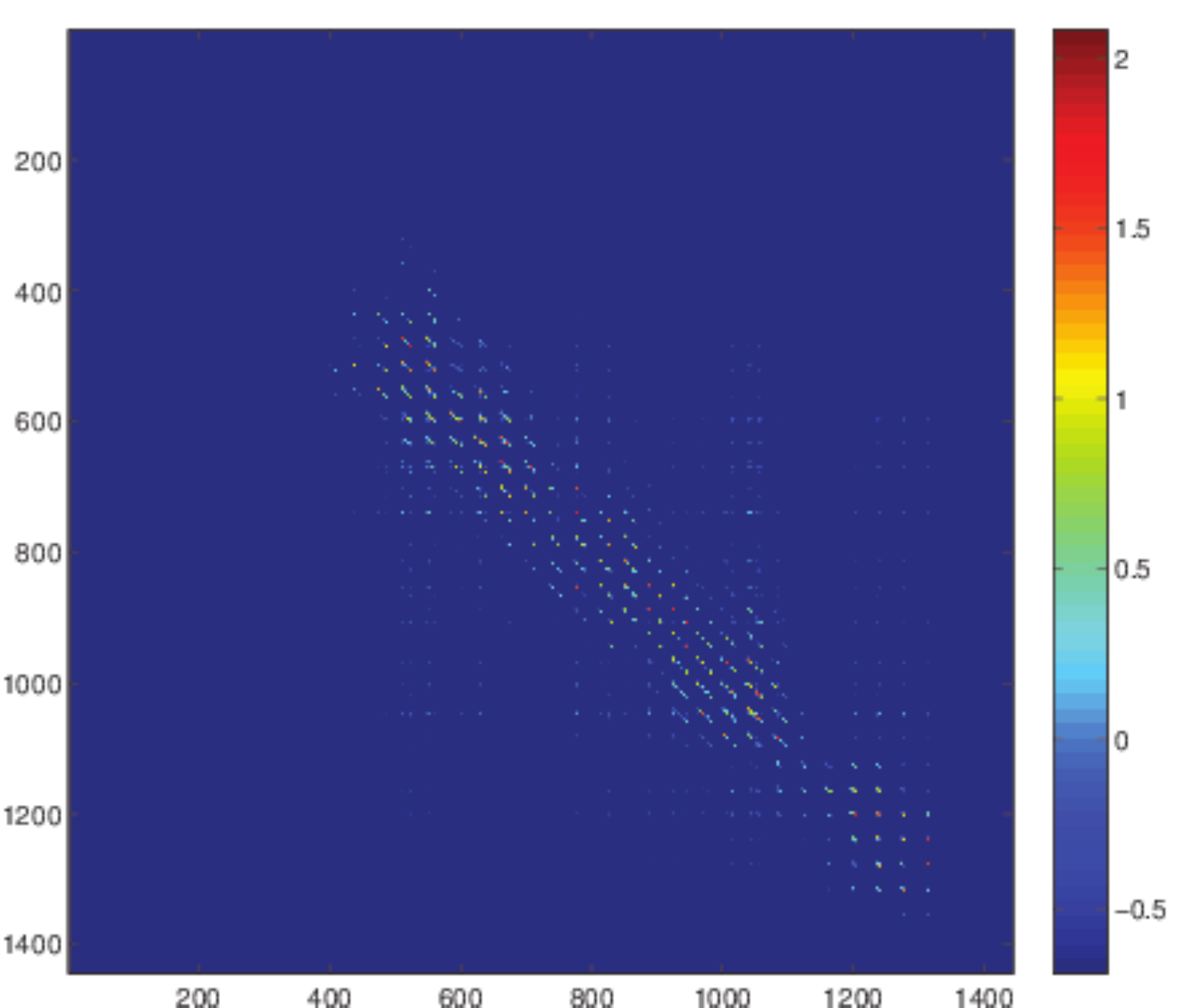}}
  \subfigure[Mask $\mbf M$]{\includegraphics[width=6.5cm]{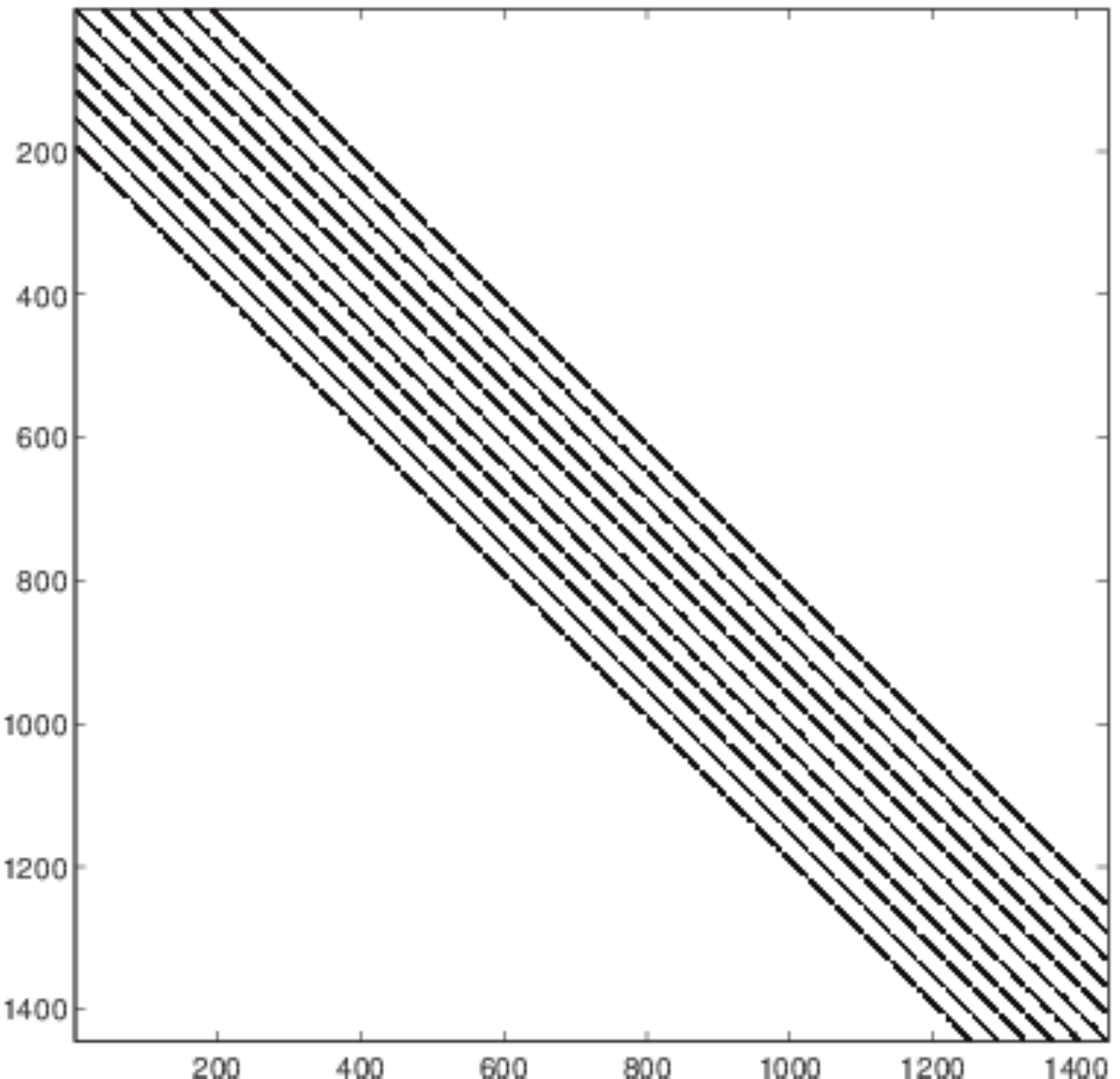}}
  \caption{(a): Wavelet coefficient matrix $\mX$ of a flower-shaped target with $L=-5$ computed
    using Daubechies wavelet of order 6. Image is the amplitude of coefficients in logarithmic
    scale. Only the first $5\permil$ of largest coefficients (in magnitude) are shown, and the
    relative error of the $N$-term approximation is $\sim 3\%$. The number of wavelets contributing
    to $\Omega$ is $38\times 38$. (b): Diagonal mask $\mbf M$ for the estimation of $\mX$ (11 band
    diagonals). There are $\sim 7\%$ non-zeros in $\mbf M$ and the relative error of approximation
    of $\mX$ with this mask is $\sim 3\%$.}
  \label{fig:diag_local}
\end{figure}

\paragraph{Localization of $\p D$}
\label{sec:localization-prop}
Numerical evidence further suggests that the strongest coefficients appear around the major band
diagonal of $\set{\TauD(\phiLn, \phiLnp)}_{n,n'}$, \ie, when $n$ and $n'$ are close. Therefore, a
large value in the diagonal coefficient ${\TauD(\phiLn, \phiLn)}$ indicates the presence of $\p D$
in the support of $\phiLn$. We plot the $815$-th column of $\mX$ in
Figure~\ref{fig:diag_local_one_row}, which correspond to the interaction between $\phiLn, n=[17,21]$
and all others $\phiLnp, n'\in\Z^2$.  We call this the \emph{localization} property of $\mX$. It
indicates that the operator $\lKstari D$ can loosely preserve the essential support of a localized
$L^2(\p D)$ function. The next proposition gives a qualitative explanation when 
$D$ is the unit
disk. 

\begin{prop}\label{prop:localization-p-d}
  Let $D$ be a unit disk. As $L\rightarrow -\infty$ we have
  \begin{align}
    \label{eq:bound_coeff_disk}
    \abs{\TauD(\phiLn, \phiLnp)} =
    \begin{cases}
      O(2^{-2L}) & \text{ \rm for overlapped } \phiLn, \phiLnp,\\
      O(2^{-L}) & \text{ \rm otherwise}.
      \end{cases}
  \end{align}
\end{prop}
\begin{proof}
  For $D$ being a unit disk, one has \cite{ammari_polarization_2007}:
  \begin{align*}
    \Kstarf D f(x) = \frac{1}{4\pi}\int_{\p D} f(y)ds(y).
  \end{align*}  
  By simple manipulations one can deduce that
  \begin{align}
    \label{eq:lKstari_on_cercle}
    \lKstarif D f (x) = \lambda^{-1}(I + (\kcst - 1)\Kstar D)(f)(x).
  \end{align}
  Therefore, the bilinear form $\TauD(\psif, \psig)$ is reduced to
  \begin{align}
    \label{eq:WPT_on_cercle}
    \int_{\p D} \psig \lKstarif D {\ddn {\psif}} ds = \frac 1 \lambda \int_{\p D}
    \psig\ddn{\psif}ds + \frac{\kcst -1}{4\pi\lambda}\int_{\p D} \psig\, 
    ds \int_{\p D}{\ddn {\psif}} ds.
  \end{align}
  Taking $\psif=\phiLn$ and $\psig=\phiLnp$ as $L\rightarrow -\infty$ the intersection between $\p
  D$ and the support of $\phiLn$ is well approximated by a line segment. 
  By a change of variables,  it
  follows that
  \begin{align}
    \label{eq:boundary_int_1}
    \intpd \phiLnp ds \intpd \p_\nu\phiLn ds = O(2^{-L}).
  \end{align}
  Similarly, one has
  \begin{align}
    \label{eq:boundary_int_2}
    \intpd \phiLnp\p_\nu\phiLn ds =
    \begin{cases}
      O(2^{-2L}) & \text{ for overlapped } \phiLn, \phiLnp,\\
      0 & \text{ otherwise}.
    \end{cases}
  \end{align}
  Substituting \eqref{eq:boundary_int_1} and \eqref{eq:boundary_int_2}  into
  \eqref{eq:WPT_on_cercle} yields \eqref{eq:bound_coeff_disk}.
\end{proof}
Hence, as $L\rightarrow -\infty$, the coefficients of the main diagonal behave like $O(2^{-2L})$,
and dominate the other band diagonals that behave like $O(2^{-L})$.

\begin{figure}[htp]
  \centering
  \subfigure[]{\includegraphics[width=7.5cm]{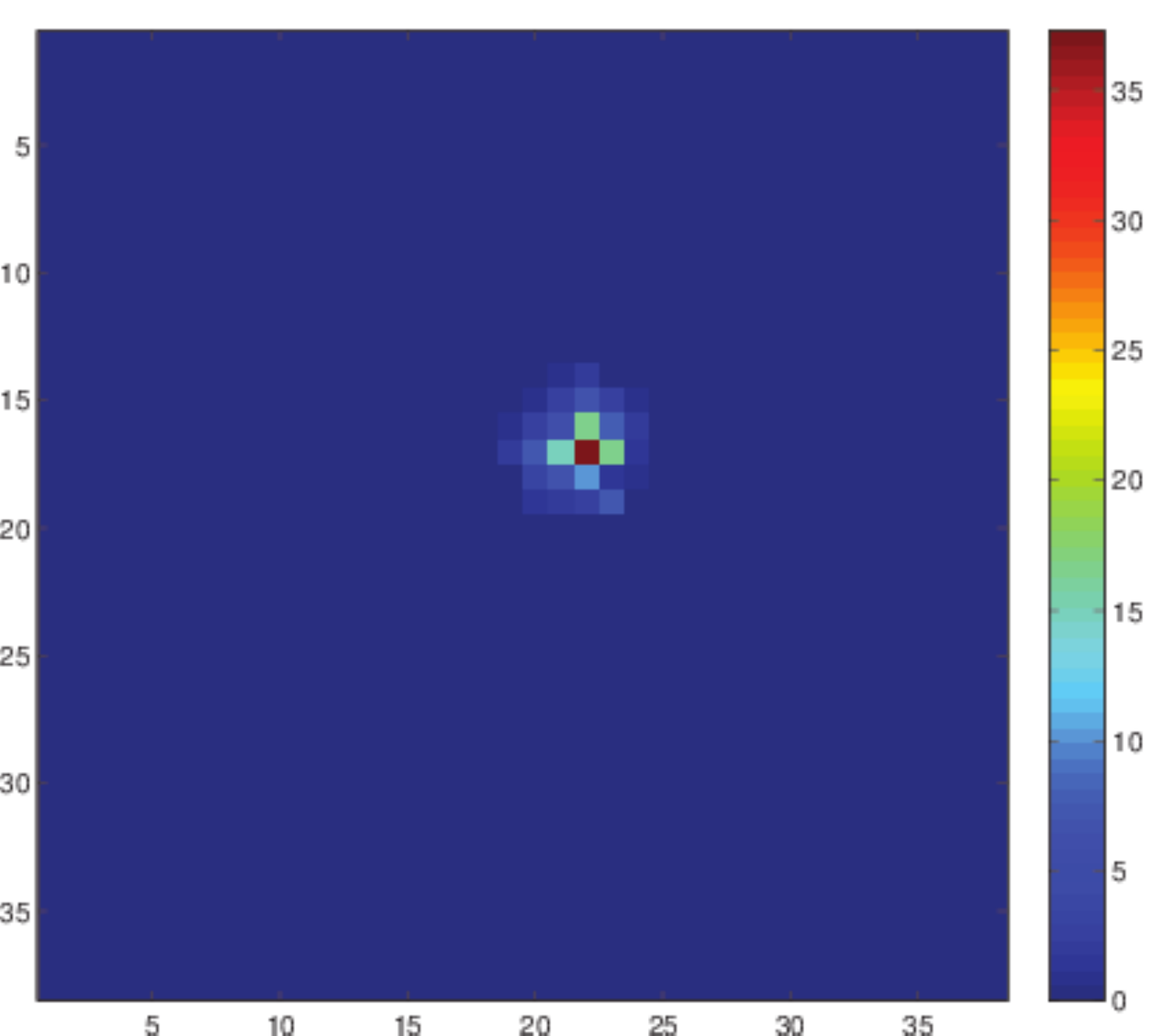}}
  \subfigure[]{\includegraphics[width=7.5cm]{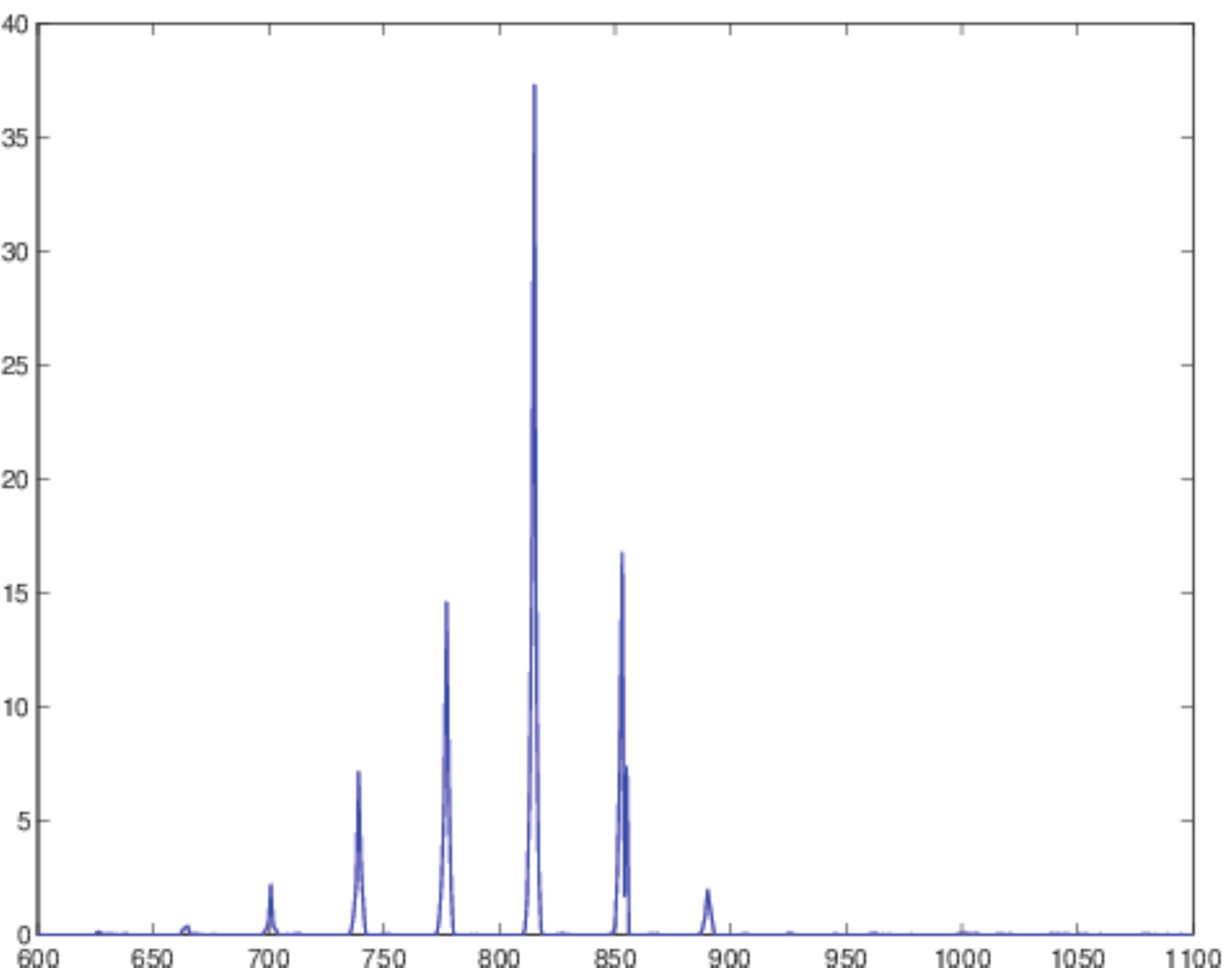}}
  \caption{Amplitude of $\set{\TauD(\phi_{L,n}, \phi_{L,n'})}_{n'}$ with $L=-4, n=[17,21]$ (or
    equivalently the 815-th row in Figure~\ref{fig:diag_local} (a)). The support of the wavelet
    $\phi_{L,n}$ intersects $\p D$. (a): View as an image with each pixel corresponding to one
    $n'$. (b): Amplitude as a function of the position index $n'$. The highest peak appears around
    $n'=815$.}
  \label{fig:diag_local_one_row}
\end{figure}

\subsection{Wavelet based imaging algorithms}
\label{sec:WPT_imaging_algo}

The localization property of $\mX$ can be used to visualize the target $D$. A simple algorithm,
called \emph{imaging by diagonal}, consists in taking the diagonal of $\mX$ (\ie, the coefficients
$\TauD(\phiLn, \phiLn)$) and reshaping it to a 2D image. Then the boundary of $\p D$ can be read off
from the image.

A drawback of this method is that the generated boundary has low resolution.  In fact, any $\phiLn$
touching the boundary $\p D$ is susceptible to yield a numerically non-negligible value of
$\TauD(\phiLn, \phiLn)$. Hence, larger is the support of the wavelet, more are the wavelets
intersecting $\p D$ and lower is the resolution. An improved method consists in searching for each
index $n\in\Z^2$, the index $n'$ maximizing the interaction between $\phiLn, \phiLnp$:
\begin{align}
  \label{eq:max_imaging_mn}
  n' = \argmax_{n'\in\Z^2} \abs{\TauD(\phiLn, \phi_{L,n'})},
\end{align}
and then accumulating $\abs{\TauD(\phiLn, \phi_{L,n'})}$ for the index $n$. We name this method
\emph{imaging by maximum}. It is higher in resolution since the effect of the wavelets touching
merely $\p D$ is absorbed by their closest neighbors lying on $\p D$. The procedure is 
summarized in
Algorithm~\ref{algo:max-imaging}.  Figures~\ref{fig:imaging_true_recon0} (a) and (b) show a comparison
between these two algorithms.

\begin{algorithm}[H]
  \caption{Imaging by maximum}
  \label{algo:max-imaging}
  \begin{algorithmic}
    \STATE Input: the matrix $\mX$ of an unknown shape $D$, a zero-valued matrix $I$.
    \FOR {$n \in\Z^2$}
    \STATE 1. $n' \leftarrow \argmax_{n' \in\Z^2} \ \abs{\mX_{n,n'}}$
    \STATE 2. $I(n) \leftarrow I(n) + \abs{\mX_{n,n'}}$;
    \ENDFOR
    \STATE Output: the 2D image $I$.
  \end{algorithmic}
\end{algorithm}

\graphicspath{{./figures/Imaging/}}

\begin{figure}[htp]
  \centering
  \subfigure[]{\includegraphics[width=7.5cm]{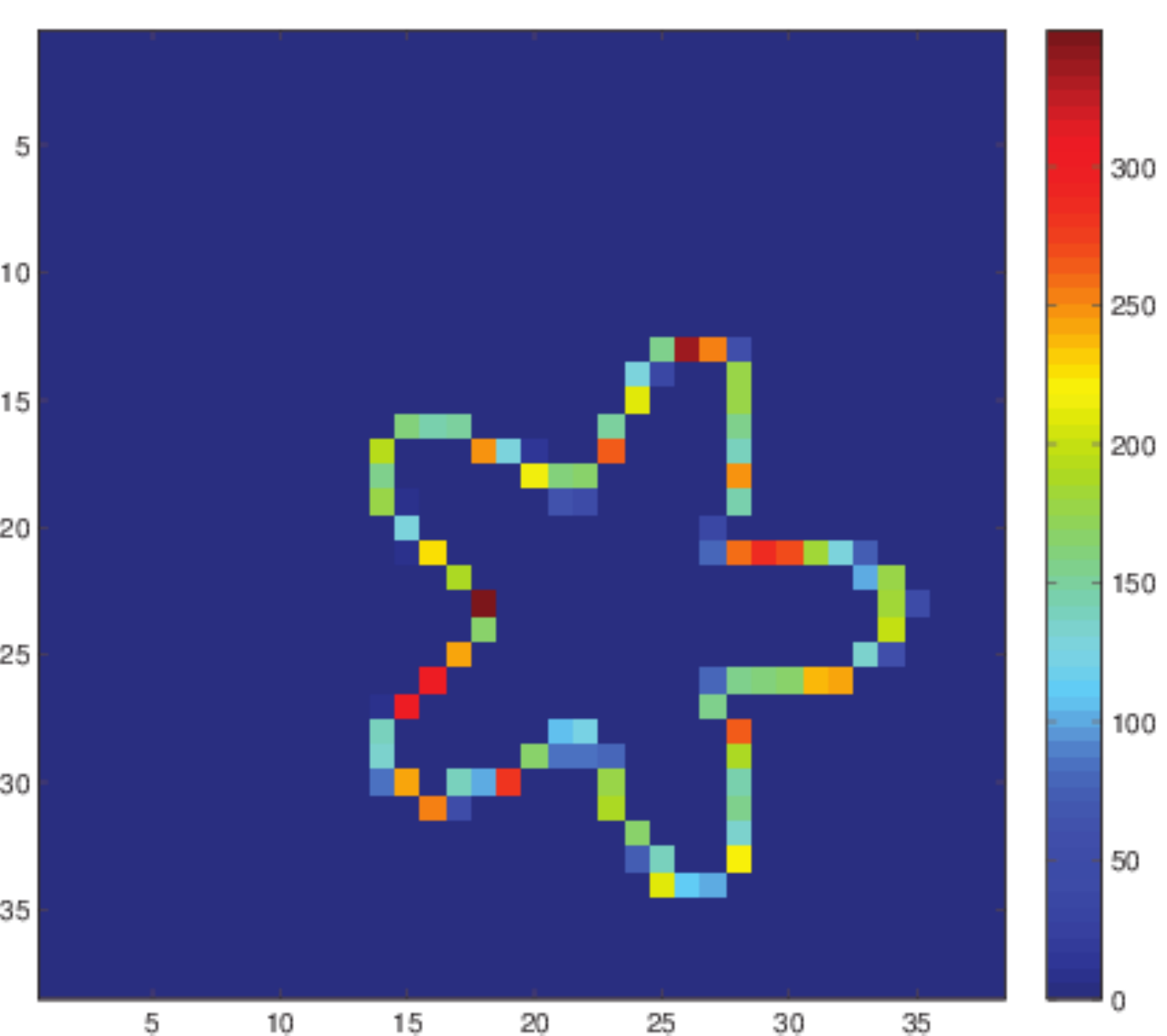}}
  \subfigure[]{\includegraphics[width=7.5cm]{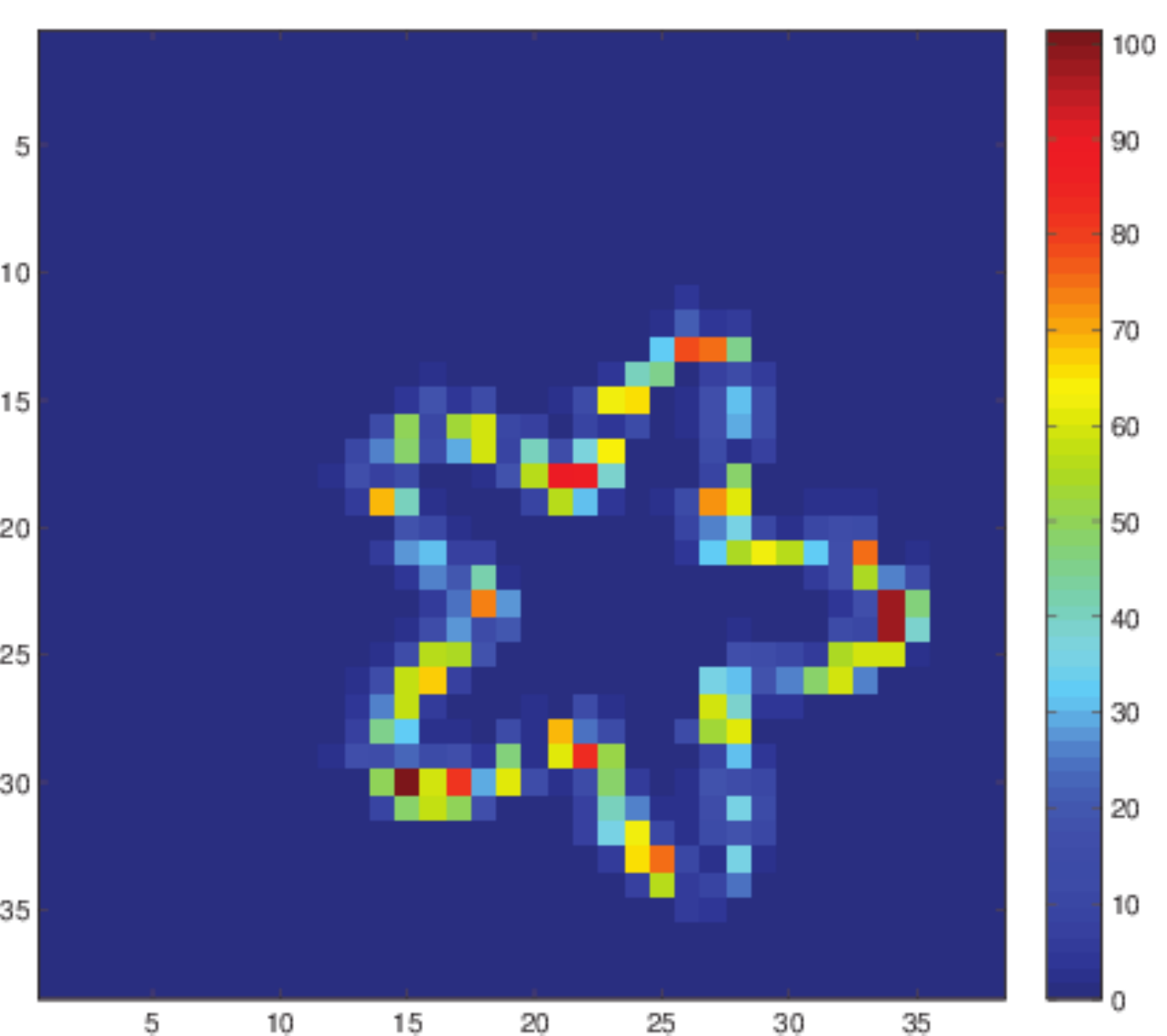}}
  \caption{Images obtained by the imaging algorithms \ref{sec:WPT_imaging_algo} applied on the true
    coefficients $\mX$ with $L=-4$. (a): Imaging by maximum. (b): Imaging by diagonal.}
  \label{fig:imaging_true_recon0}
\end{figure}

\subsection{Reconstruction of $\mX$ by $\ell^1$ minimization}
\label{sec:reconstr-mx-ell1}

As the scale $L$ decreases, the dimension of $\mX$ increases rapidly as $\propto 2^{-4L}$. On the
other hand, the band diagonal structure of $\mX$ shows that the largest coefficients distribute on
the major band diagonals, which is an important \emph{a priori} information allowing to reduce
considerably the dimension of the unknown to be reconstructed.

For this, we fix \emph{a priori} $N_0>0$ and assume that the coefficient $\TauD(\phiLn, \phiLnp)$
can be neglected when $\abs{n-n'}>N_0$. We construct accordingly a band diagonal mask $\mbf M$
taking values $0$ or $1$ by choosing $N_0$ proportional to the support size of the wavelet. Remark
that the mask constructed in this way is not adaptive and does not contain any information about the
boundary of the target. Figure~\ref{fig:diag_local} (b) shows a mask $\mbf M$ with $N_0=5$.

Given the high dimension of $\mX$ and its sparsity, we seek a sparse solution by solving the
$\ell^1$ minimization problem as follows \cite{mallat_wavelet_2008, chen_atomic_1998,
  candes_near-optimal_2006}:
\begin{align}
  \label{eq:l1_estim}
  \Xest:=\argmin_{\mX} \Norm{\mbf{L(M\odot X)-V}}^2 + \mu \Norm{\mbf M \odot \mX}_{1,w},
\end{align}
where $\mbf M$ is the band diagonal mask (Figure~\ref{fig:diag_local} (b)), $\odot$ denotes the
termwise multiplication, $\mu>0$ is the regularization parameter, and $\Norm{\mbf x}_{1,w} = \sum_n
w_n\abs{x_n}$ is the reweighted $\ell_1$ norm. We set the weight $w$ in such a way that the operator
$\mbf L$ is normalized columnwisely. The constant $\mu$ is determined by the universal threshold
\cite{mallat_wavelet_2008} (tuned manually to achieve the best result if necessary):
\begin{align}
  \label{eq:mu_universal_th}
  \mu \propto \snoise\sqrt{N_sN_r}\sqrt{2\log \norm{\mbf M}_1}
\end{align}
with $\norm{\mbf M}_1$ being the number of non-zero values in $\mbf M$. Problem \eqref{eq:l1_estim}
admits a unique sparse solution \cite{mallat_wavelet_2008} under appropriate conditions, and can be
solved numerically via efficient algorithms; see for example \cite{beck_fast_2009}.

\section{Numerical experiments}
\label{sec:num_exp} 

In this section we present some numerical results to illustrate the efficiency of the wavelet
imaging algorithm proposed in section~\ref{sec:WPT_imaging_algo}.  The wavelet used here is the
Daubechies wavelet of order 6. We consider a near-field measurement system
(Figure~\ref{fig:acq_systems} (b)) with $20\times 20$ uniformly distributed sources and receivers.

\subsection{Parameter settings}
\label{sec:parameter-settings}

We set the conductivity constant $\kcst=4/3$ and use a flower-shaped target as $D$. The whole
procedure of the experiment is as follows. First, the data $V_{sr}$ are simulated by evaluating
\eqref{eq:repr_solution} for all sources $x_s$ and receivers $y_r$ of the measurement system. A
white noise of level
\begin{align}
  \label{eq:snoise_V}
  \snoise=\sigma_0\norm{\V}_F/\sqrt{N_s N_r}
\end{align}
is added to obtain the noisy data $\V$, with $\sigma_0$ being the percentage of noise in
data. Thereafter the minimization problem \eqref{eq:l1_estim} is solved with the parameters and
methods described in section~\ref{sec:reconstr-mx-ell1}. Finally, from the reconstructed
coefficients $\mX$, we apply Algorithm~\ref{algo:max-imaging} to obtain a pixelized image of $\p D$.

The vectors $\bgammas, \bgammar$ in \eqref{eq:Vsr_bgamma_E} contain the wavelet coefficients of
the Green functions $\Gammas$ and $\Gammar$ respectively. They are computed by first sampling $\Gammas$
and $\Gammar$ on a fine Cartesian lattice of sampling step $\lesssim 2^{-11}$ on $\Omega$ (the
singularity point $x_s$ is numerically smoothed) and then applying the discrete fast wavelet
transform on these samples to obtain the coefficients at the desired scale.

\subsection{Results of the imaging algorithm}
\label{sec:numerical-results}
Figures~\ref{fig:imaging_recon_noisy} (a, b) show the results of imaging obtained at the scale $L=-4$
with different noise levels.  It can be seen that even in a highly noisy environment (\eg
$\sigma_0=100\%$) the boundary of $D$ can still be correctly located.

Remark that for the near-field internal measurement system Figure~\ref{fig:acq_systems} (b), one can
obtain an image of $\p D$ directly from the data $\V$ (in fact, $V_{sr}$ being defined by
\eqref{eq:repr_solution} has large amplitude if $x_s$ and/or $y_r$ is close to $\p D$). Nonetheless,
such a direct imaging method is far less robust to noise than the wavelet based algorithm and its
resolution is limited by the density of the transmitters, as shown in
Figures~\ref{fig:imaging_recon_noisy} (c, d).

\graphicspath{{./figures/Imaging/}}

\begin{figure}[htp]
  \centering
  \subfigure[]{\includegraphics[width=7.5cm]{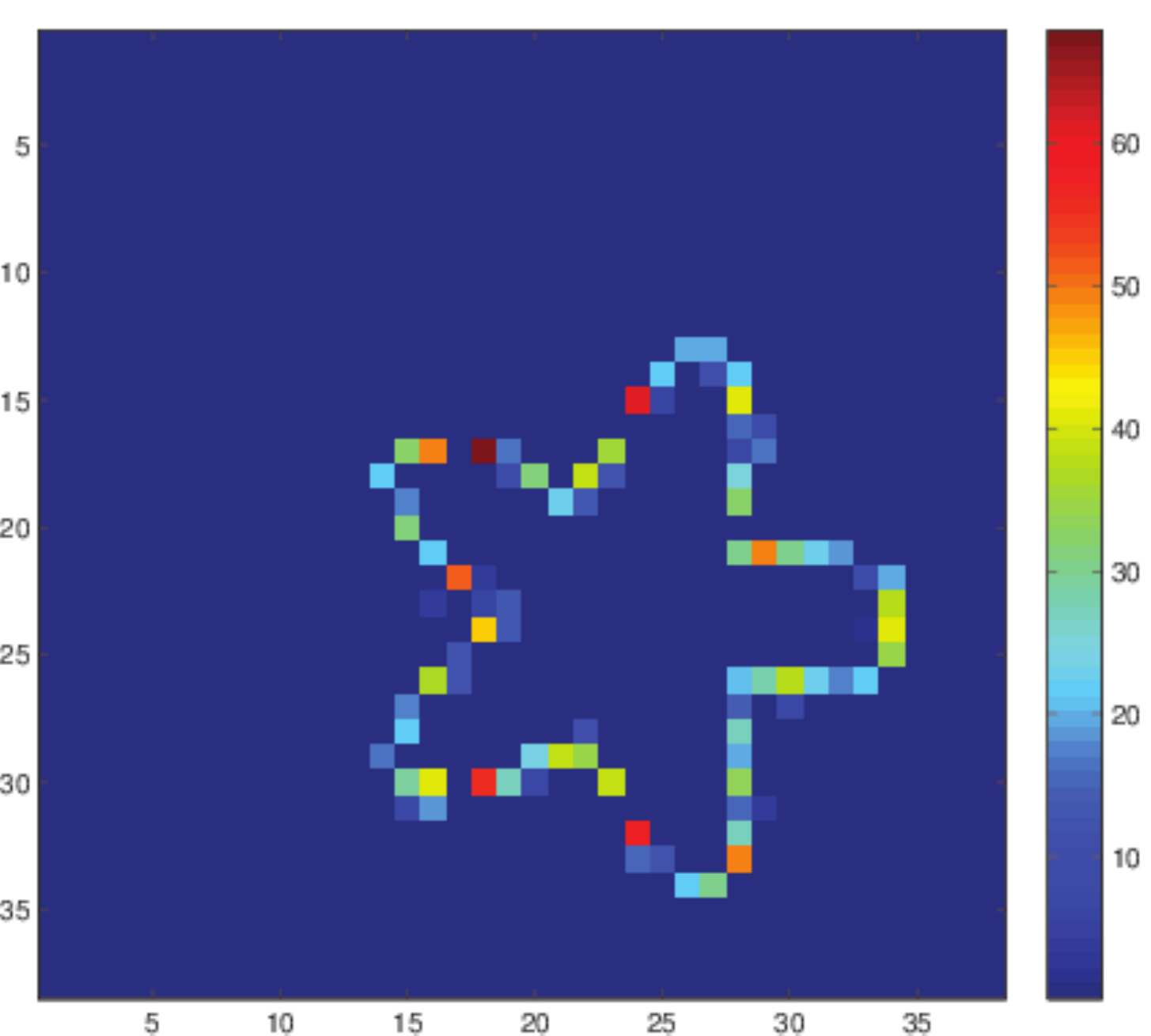}}
  \subfigure[]{\includegraphics[width=7.5cm]{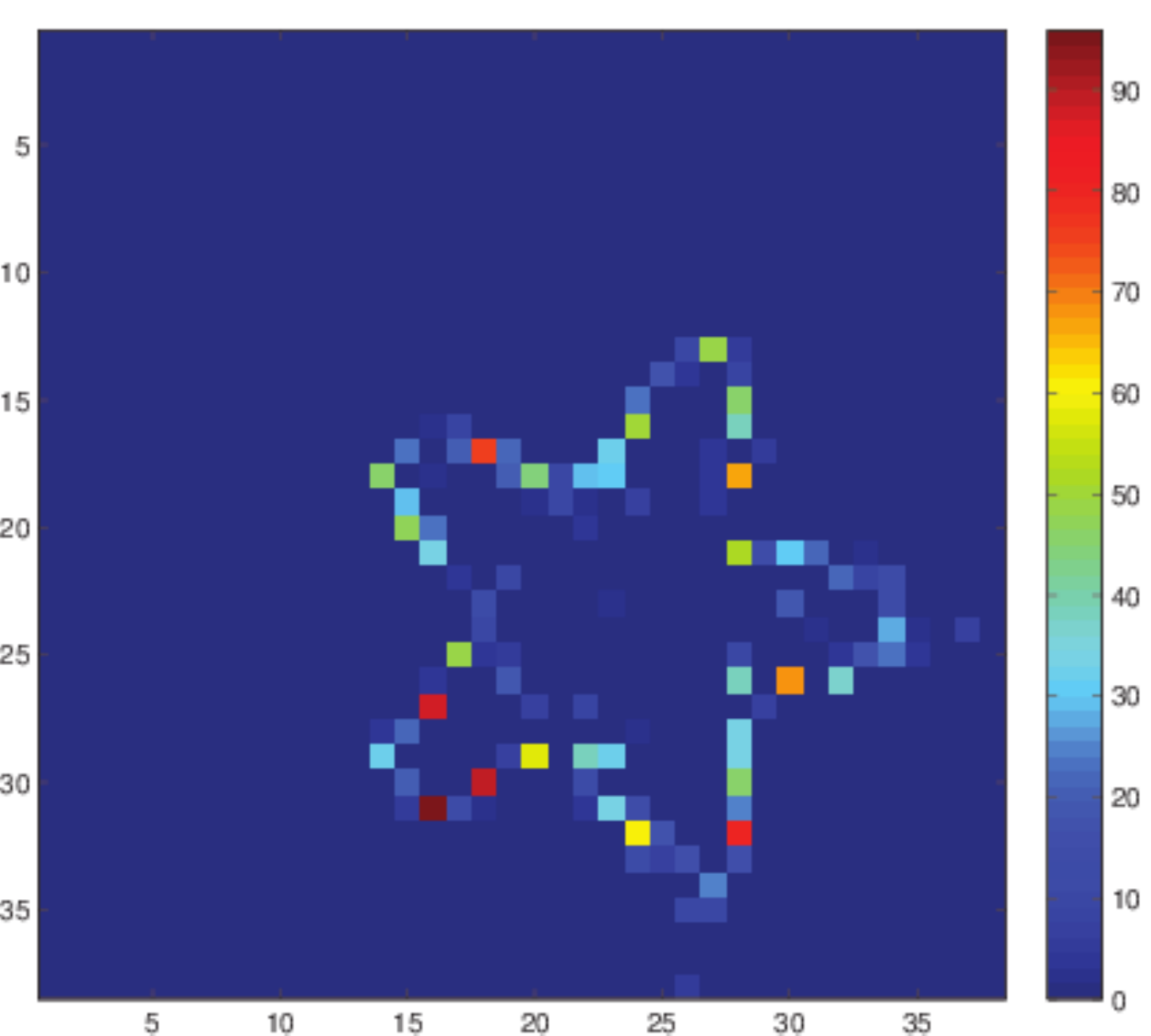}}
  \subfigure[]{\includegraphics[width=7.5cm]{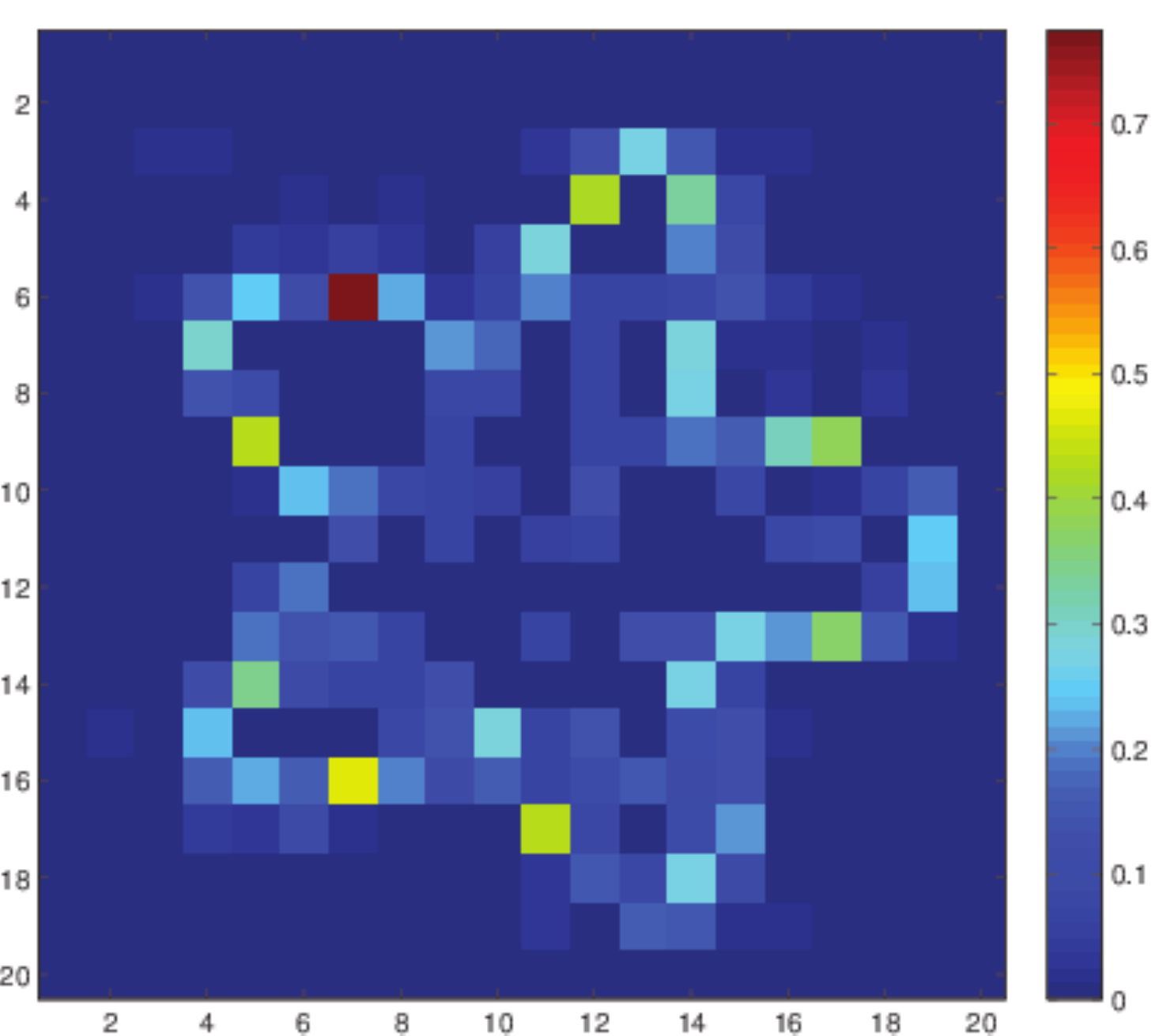}}
  \subfigure[]{\includegraphics[width=7.5cm]{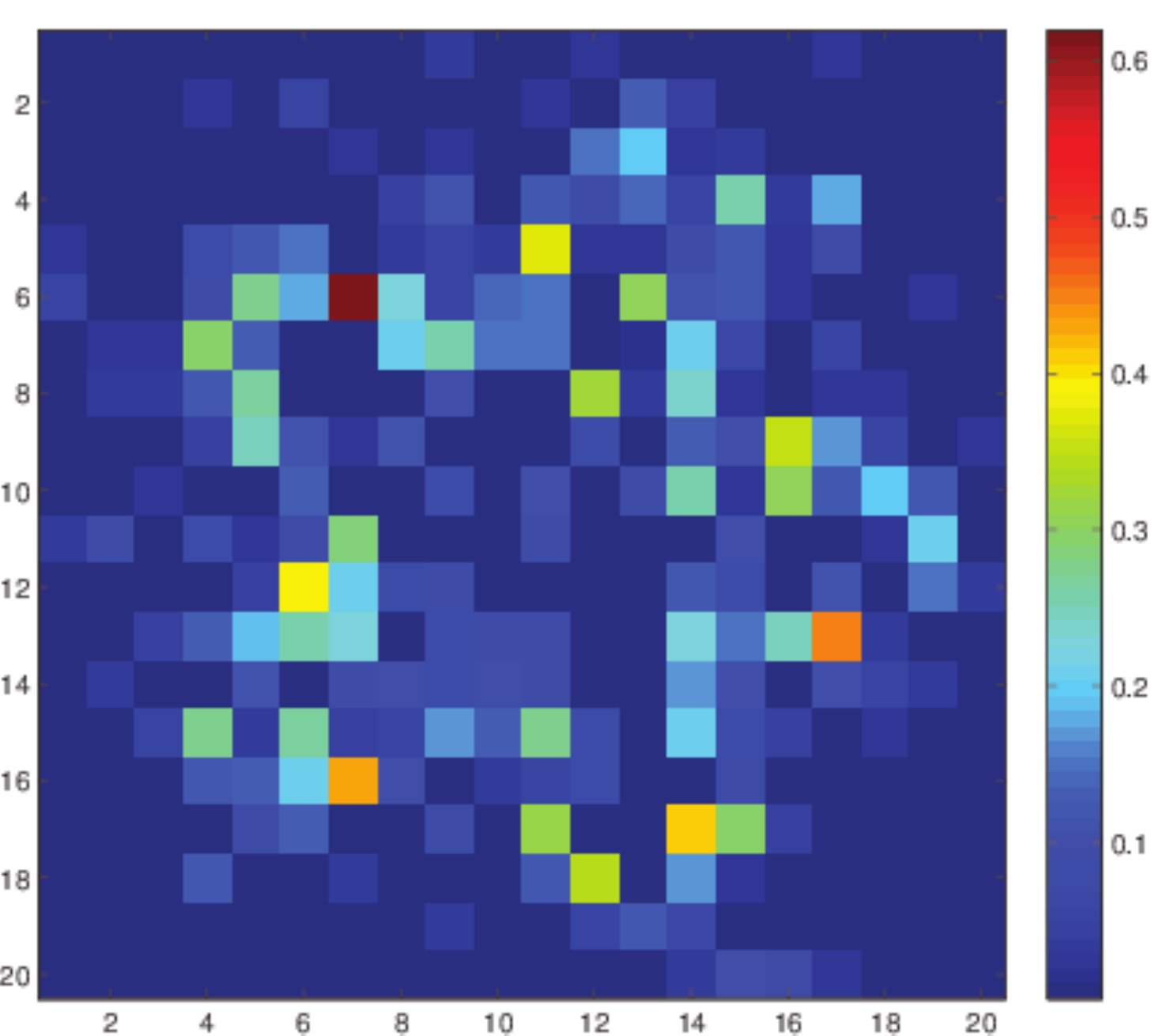}}
  \caption{(a, b): Images obtained by Algorithm~\ref{algo:max-imaging} applied on $\mX$
    reconstructed from data with noise, $L=-4$. (c, d): Images obtained directly from data $\V$.
    (a, c): Noise level $\sigma_0=50\%$, (b, d): $\sigma_0=100\%$.}
  \label{fig:imaging_recon_noisy}
\end{figure}

In Figure~\ref{fig:imaging_recon_noisy_super} the same experiments of imaging with noisy data were
conducted at the scale $L=-5$. We notice that the images returned by
Algorithm~\ref{algo:max-imaging} have dimension $64\times 64$, which is much higher than that of the
grid of transmitters ($20\times 20$). Furthermore, the results remain robust up to the noise level
$\sigma_0=50\%$. These confirm the super-resolution character of the wavelet based imaging
algorithm.

\begin{figure}[htp]
  \centering
  \subfigure[]{\includegraphics[width=7.5cm]{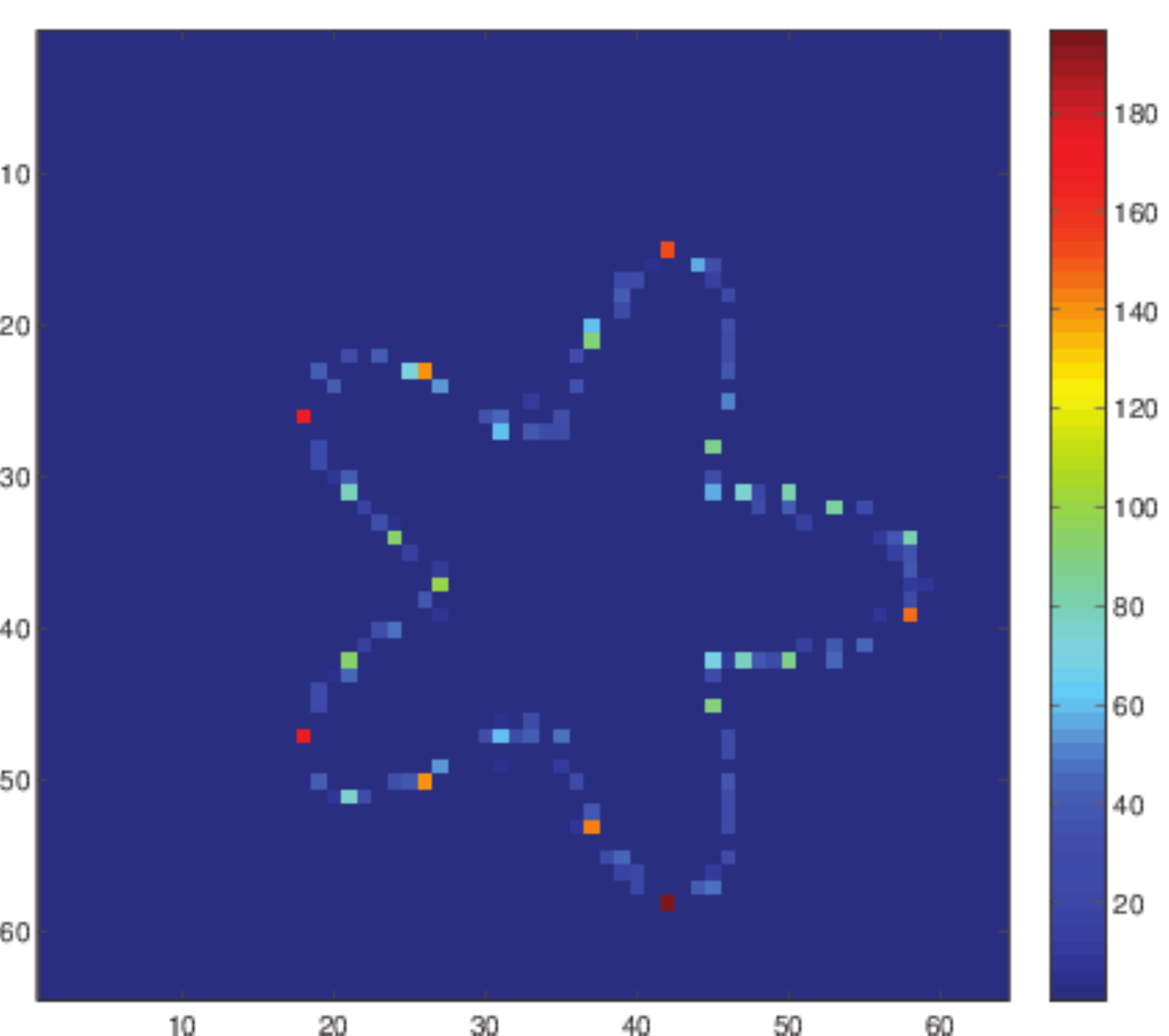}}
  \subfigure[]{\includegraphics[width=7.5cm]{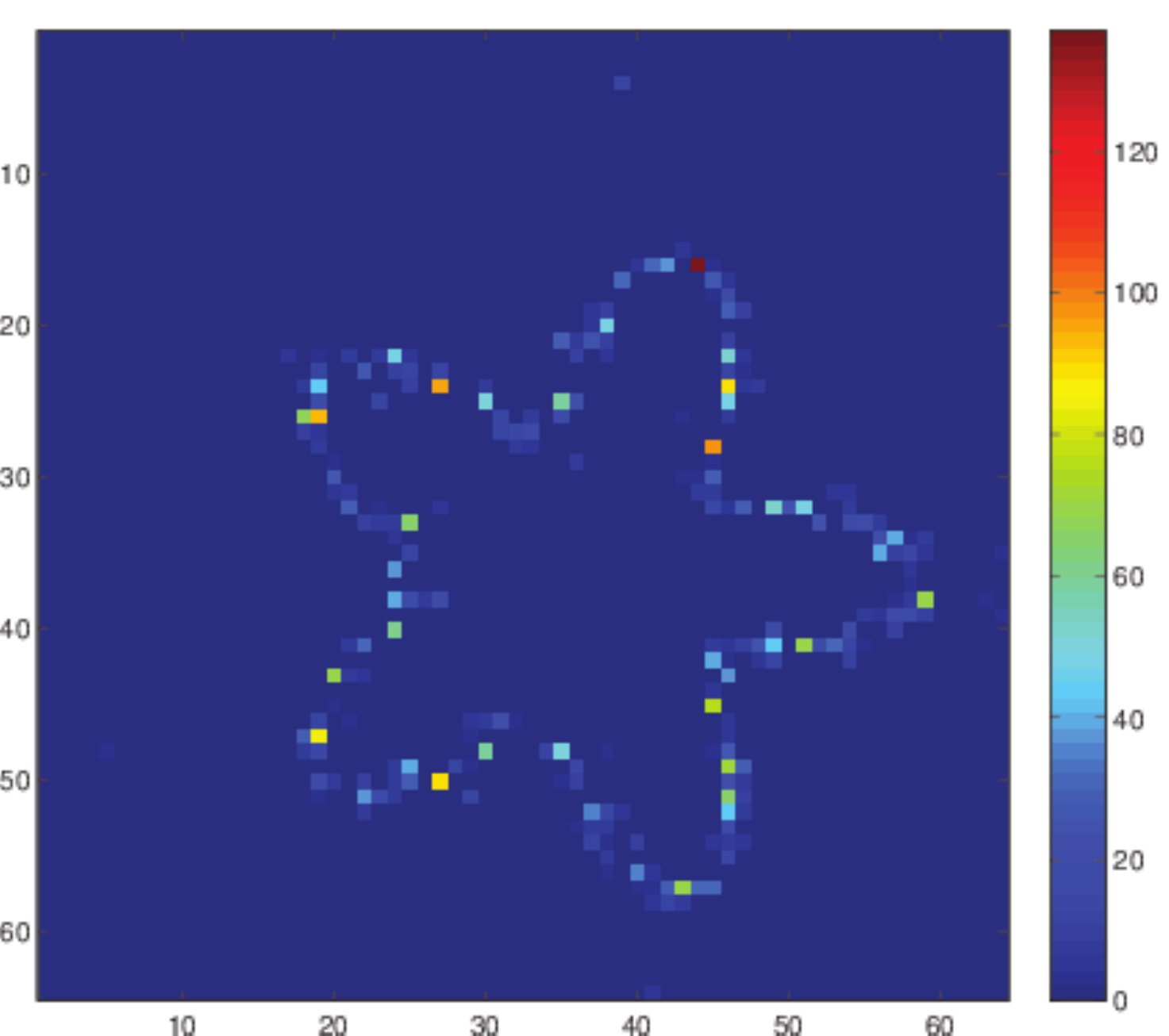}}
  \caption{Same experiments as Figures~\ref{fig:imaging_recon_noisy} (a, b) at the scale $L=-5$. (a):
    Imaging by maximum with $\sigma_0=10\%$ of noise, (b): with $\sigma_0=50\%$ of noise.}
  \label{fig:imaging_recon_noisy_super}
\end{figure}

\section{Discussion}
\label{sec:discussion}

\subsection{Effect of the conductivity constant $\kcst$}
\label{sec:effet-cond-const}
The constant $\kcst$ as defined in section~\ref{sec:electric-sensing} is actually the ratio between
the conductivity of the target and the background (set to 1 in this paper). Further numerical
experiments suggest that the performance of Algorithm~\ref{algo:max-imaging} depend on $\kcst$: the
results may deteriorate when $\kcst$ becomes large (\eg~$\kcst\geq 4$). This can be explained easily
for the case of a unit disk. In fact, it can be seen from \eqref{eq:WPT_on_cercle} that the ratio
between the overlapped and non overlapped (in terms of the functions $\phiLn, \phiLnp$) coefficients
of $\mX$ varies with $\kcst$ as $1/(\kcst-1)$. Hence, the localization property
(section~\ref{sec:localization-prop}) becomes more (resp. less) pronounced when $\kcst\rightarrow 1$
(resp. $\kcst\rightarrow +\infty$), and the imaging algorithm is impacted accordingly. Nonetheless,
we note also that when $\kcst\rightarrow 1$, the target $D$ becomes indistinguishable from the
background and the measured data $\V$ decreases to zero (without considering the noise). These
observations suggest that in practice there may exist some numerical ranges for $\kcst$ and for the
noise level on which the imaging algorithm is more or less effective.

\subsection{Representation with the wavelet $\psi$}
\label{sec:repr-with-wavel}

In section~\ref{sec:WPT} we used only the scaling function $\phi$ for the approximation and the
representation of $\TauD$, while it is also possible to use the wavelet functions $\psi^k$ together
with $\phi$ to represent $\TauD$ and obtain another form of $\mX$.  More precisely, let $W_j$ be the
detail space spanned by $\set{\psikjn}_{n\in\Z^2,k=1,2,3}$ and $Q_j$ be the orthogonal projectors
onto $W_j$. For any two scales $L \leq J$, it holds
\begin{align}
  \label{eq:Tau_apprx_std}
  \TauD(P_Lf, P_L g) = &\sum_{j,j'=L+1}^J \myubrace{\TauD(Q_j f, Q_{j'}g)}{D_{j,j'}} +
  \sum_{{j'}=L+1}^J \myubrace{\TauD(P_J f, Q_{j'}g)}{C_{J,j'}} \notag \\ & + \sum_{j=L+1}^J
  \myubrace{\TauD(Q_jf, P_Jg)}{B_{j,J}} + \myubrace{\TauD(P_J f, P_J g)}{A_{J,J}} = \mbf f^\top \mX
  \mbf g,
\end{align}
where we used the fact that $P_{L} f = P_{J} f + \sum_{j=L+1}^{J} Q_j f$, and $\mbf f, \mbf g$ are respectively
the coefficient vectors of $f$ and $g$  under the basis $\set{\phi_{J,n}, n\in\Z^2} \cup
\set{\psikjn, j=L\ldots J, n\in\Z^2, k=1,2,3}$. The coefficient matrix $\mX$ now takes the form
\begin{align}
  \label{eq:WPT_std_2}
  \mX = \mX[D,L,J] =
  \begin{pmatrix}
    D_{L+1,L+1} & \ldots & D_{L+1,J} & B_{L+1,J}\\
    %D_{L+2,L+1} & D_{L+2,L+2} & \ldots & D_{L+2,J} & B_{L+2,J}\\
    \vdots & \ddots & \vdots & \vdots \\
    % & & \vdots & & \\
    D_{J,L+1} & \ldots & D_{J,J} & B_{J,J}\\
    C_{J,L+1} & \ldots & C_{J,J} & A_{J,J}\\
  \end{pmatrix},
\end{align}
where $D,C,B,$ and $A$ denote the block matrices corresponding to the terms marked by braces in
\eqref{eq:Tau_apprx_std} respectively. In particular, $D_{j,j'}$ contains the detail coefficients of
type $\TauD(\psikjn,\psikjnp)$ with $k,k'=1,2,3$, while $A_{J,J}$ contains the approximation
coefficients $\TauD(\phiJn,\phiJnp)$. Remark that in the case $L=J$, $\mX[D, L, L]$ is reduced to
$A_{L,L}$ which is identical to the coefficient matrix defined in \eqref{eq:A_mat}.

Moreover, one can easily prove (using the conjugated filters) that for any $J,J'\geq L$, $\mX[D, L,
J']$ and $\mX[D, L, J']$, regarded as $\ell^2(\Z)$ vectors, are equivalent up to an $\ell^2$ unitary
transform. Therefore the choice of the scale $J$ is not important since $\mX[D, L, J]$ is equivalent
to $\mX[D,L,L]$ for any $J\geq L$, and their dimensions are asymptotically equal as $L\rightarrow
-\infty$ for a fixed domain $\Omega$.

A natural question is to know whether the equivalent representation $\mX[D,L,J]$ with $J<L$ is more
sparse than $\mX[D,L,L]$. In Figure~\ref{fig:DD_AA_decay} we plot the decay of the coefficients of the
four block matrices in $\mX[D,L,L+1]$ with $L=-5$. It can be seen that for the numerical range
considered here, the detail coefficients have similar decay as the approximation coefficients. In
fact, like $\TauD(\phiLn,\phiLnp)$, the main reason for the sparsity of the detail coefficients
$\TauD(\psikjn, \psikjnp)$ is the intersection between the support of wavelets and the boundary $\p
D$, and for the same reason the localization property (\ie, Proposition \ref{prop:localization-p-d})
remains valid for the wavelets $\psikjn, k=1,2,3$. Hence the representation $\mX[D,L,J]$ has a
similar sparsity as $\mX[D,L,L]$ and does not present substantial advantages for the applications
considered in this paper.

\graphicspath{{./figures/}}

\begin{figure}[htp]
  \centering
  \subfigure[]{\includegraphics[width=7.5cm]{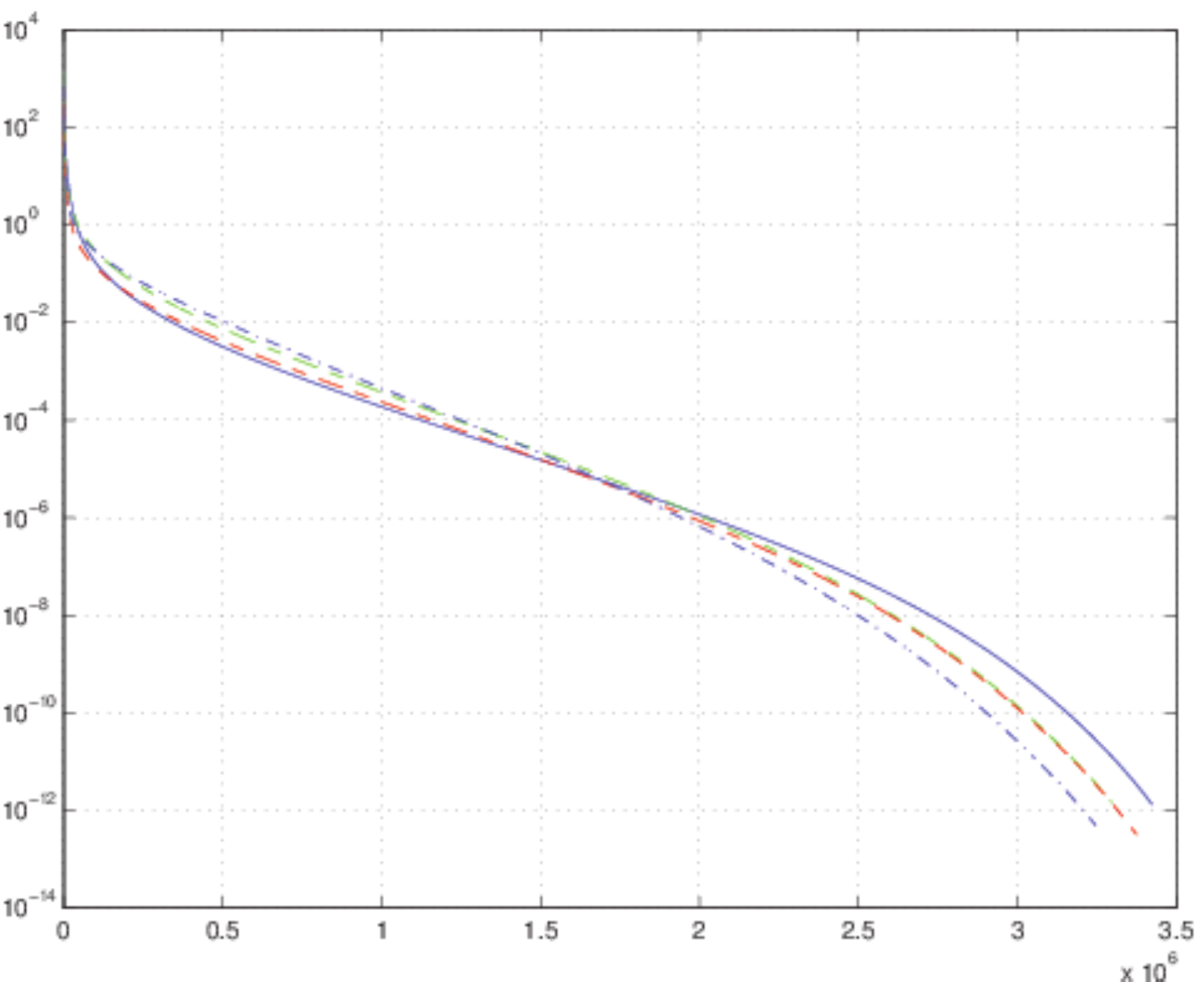}}
  \subfigure[]{\includegraphics[width=7.5cm]{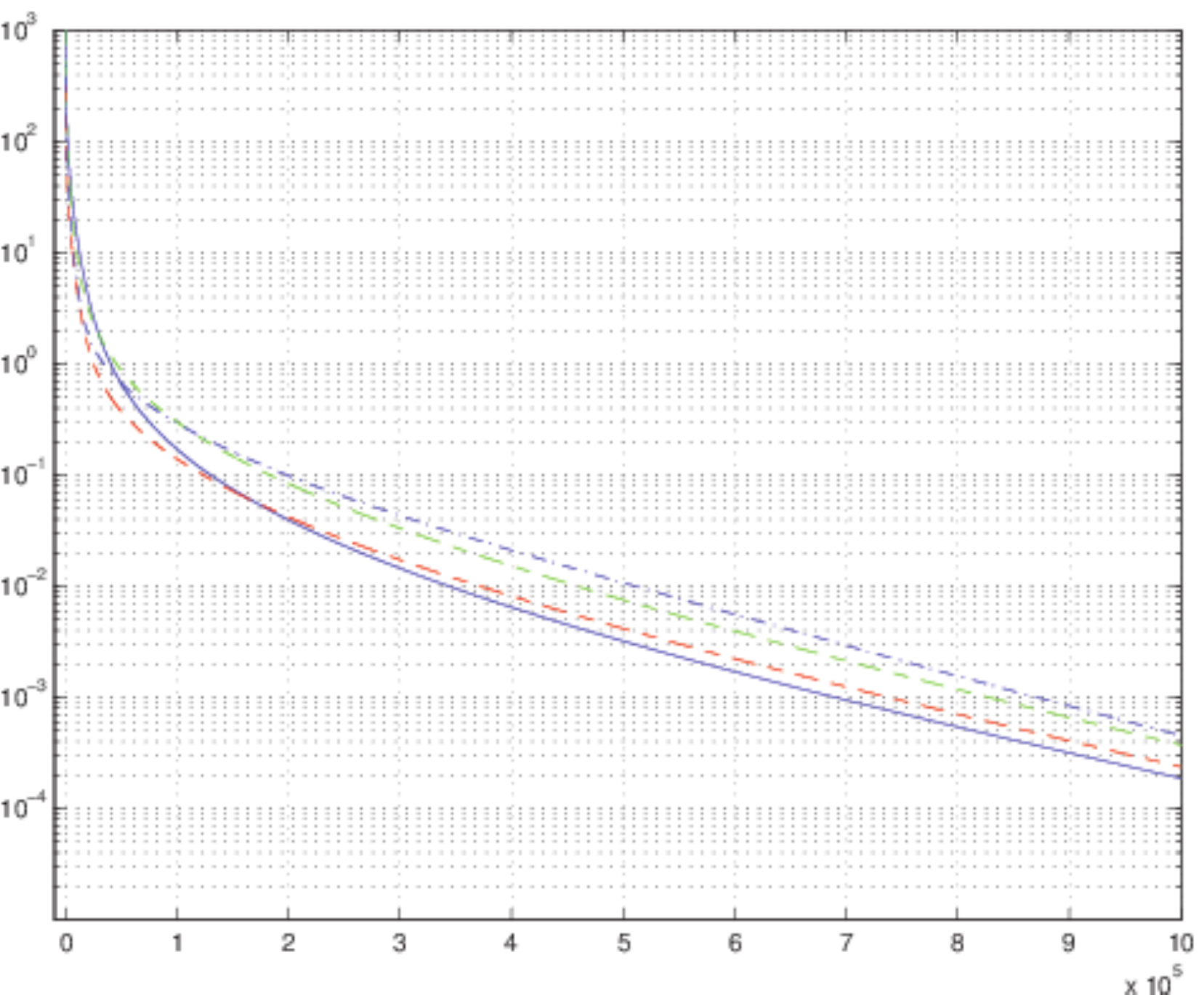}}
  \caption{Coefficients (absolute value) of $\mX[D,L,L+1]$ with $L=-5$ in decreasing order in the
    logarithmic scale. (a): All coefficients. (b): Zoom-in on the first $10^6$ coefficients. The
    different curves in the figures represent: $\TauD(\phiLn,\phiLnp)$ (dash-dot line),
    $\TauD(\psi^3_{L,n}, \phiLnp)$ (dashed line in green), $\TauD(\phiLn, \psi^3_{L,n'})$ (dashed
    line in red), and $\TauD(\psi^3_{L,n}, \psi^3_{L,n'})$ (solid line).}
  \label{fig:DD_AA_decay}
 
\end{figure}
\section{Conclusion}
\label{sec:conclusion}

In this paper we presented a general framework for the electro-sensing problem, and proposed a new
wavelet based approach for the solution of the inverse problem and the visualization of the target. The
new approach is complementary to the previous developed polynomial based approach in that both of
them can be seen as choosing the basis adapted to the measurement system. In case of the near-field
measurement, the wavelet approach is more appropriate than to the polynomial approach since it gives
a sparse representation of the geometric information of the target and allows to reconstruct more
information by exploiting the sparsity using  $\ell^1$ minimization, which is superior in
robustness than the linear estimator in this case.  Finally, numerical results show the performance
of the wavelet imaging algorithms, confirming the efficiency of the new approach.

\appendix

\section{Proof of Lemma \ref{lem:polynomial-basis-Hs}}
\label{sec:proof-lemma-polybasis}

\begin{proof}
  For a given $g\in H^s(\Omega)$, by density of $\mcl C^\infty(\Omega)$ in $H^s(\Omega)$, for any
  $\epsilon>0$ there exists $u\in\mcl C^\infty(\Omega)$, such that $\norm{g-u}_{\sss{H^s}}\leq
  \epsilon/2$.  On the other hand, since $\Omega$ is bounded, one can construct the Bernstein
  polynomial to approximate a $\mcl C^\infty$-function and its first $s$ order derivatives
  simultaneously and uniformly on $\Omega$ \cite{kingsley_bernstein_1951}. Hence, there exists
  $K=K(\epsilon)>0$ and a polynomial
  \begin{align}
    \label{eq:polynomial_p}
    p(x)=\sum_{\infabs\alpha\leq K} a_\alpha x^\alpha, \ x\in\R^2,
  \end{align}
  such that $\norm{u-p}_{\sss{H^s}}\leq \epsilon/2$. Therefore,
  \begin{align*}
    \norm{g-p}_{\sss{H^s}}\leq \norm{g-u}_{\sss{H^s}}+\norm{u-p}_{\sss{H^s}} \leq \epsilon,
  \end{align*}
  which proves that the polynomial basis is complete in $H^s(\Omega)$ for $s\geq 0$.
\end{proof}

% \begin{proof}
%   For a given $g\in H^s(\Omega)$, by density of $\mcl C^\infty(\Omega)$ in $H^s(\Omega)$, for any
%   $\epsilon>0$ there exists $u\in\mcl C^\infty(\Omega)$, such that $\norm{g-u}_{\sss{H^s}}\leq
%   \epsilon/2$.  On the other hand, since $\Omega$ is bounded, one can construct the Bernstein
%   polynomial to approximate a $\mcl C^\infty$ function and its first $s$ order derivatives
%   simultaneously and uniformly on $\Omega$ \cite{kingsley_bernstein_1951}. Hence there exists
%   $K=K(\epsilon)>0$ and a polynomial
%   \begin{align}
%     \label{eq:polynomial_p}
%     p(x)=\sum_{\abs\alpha\leq K} a_\alpha x^\alpha, \ x\in\R^2, 
%   \end{align}
%   such that $\norm{u-p}_{\sss{H^s}}\leq \epsilon/2$. Therefore,
%   \begin{align*}
%     \norm{g-p}_{\sss{H^s}}\leq \norm{g-u}_{\sss{H^s}}+\norm{u-p}_{\sss{H^s}} \leq \epsilon,
%   \end{align*}
%   which proves that the polynomial basis is complete in $H^s(\Omega)$ for $s\geq 0$.
% \end{proof}

%%% Local Variables: 
%%% mode: latex
%%% TeX-master: "main"
%%% End: 

\section{Proof of Proposition \ref{prop:decay_wavelet_coeff_Gammas}}
\label{sec:proof-decay-wvl-coeff}

\begin{proof}
  Since $x_s\notin\COmega$, there exists a scale $j_0$ small enough such that $x_s\neq 2^j n$ for any
  $n\in\Lambda_j^k, j\leq j_0, k=1,2,3$. The Taylor expansion up to order $p-1$ of $\Gammas$ reads:
  \begin{align*}
    \Gammas(x) = \sum_{\abs\alpha=0}^{p-1}\frac{(x-2^jn)^\alpha}{\alpha!}\p^\alpha\Gamma(2^jn-x_s)+R(x)
  \end{align*}
  with the rest $R(x)$ being given by
  \begin{align*}
    R(x) = \sum_{\abs\alpha=p} \frac{p}{\alpha!} {(x-2^jn)}^\alpha\int_0^1{(1-t)}^{p-1}.
    \p^\alpha\Gamma(2^jn-x_s+t(x-2^jn))dt
  \end{align*}
  For $k=1,2,3$, the two-dimensional wavelet $\psikjn$ is orthogonal to the polynomial $x^\alpha$
  for any $\abs\alpha<p$. Hence by the change of variables $x\rightarrow 2^j x$ we obtain
  \begin{align}
    \label{eq:Gammas_psikjn_2}
    \seq{\Gammas, \psikjn} 
    &= \sum_{\abs\alpha=p} \frac{ 2^{j(p+1)} p}{\alpha!} \int_0^1(1-t)^{p-1}\int_{\supp\psi^k} x^\alpha \psi^k(x)
    \p^\alpha\Gamma(2^jn-x_s+t 2^j x)dx\,dt.
  \end{align}
  When $j\rightarrow -\infty$, there exists positive constants $c_0$ and $c_1$ depending only on $\COmega$
  and $\psi^k$, such that for any $n\in\Lambda_j^k$, the distance between $x_s$ and $2^j n$
  satisfies 
  $$c_0 \rho\leq \norm{2^jn - x_s}\leq c_1 \rho.$$
  Combining the fact that $\psi^k$ is compactly supported together with the estimate $\Abs{\p^\alpha\Gamma(x)}\asymp
  \norm{x}^{-\abs\alpha}$, we conclude from \eqref{eq:Gammas_psikjn_2} that
  \begin{align*}
    \abs{\seq{\Gammas, \psikjn}} \asymp 2^{j(p+1)} \rho^{-p}  \ \text{ as } j\rightarrow -\infty,
  \end{align*}
  where the underlying constants depend only on $x_s, \COmega, \psi^k,$ and $p$.
\end{proof}

%%% Local Variables: 
%%% mode: latex
%%% TeX-master: "main"
%%% End: 

%\input{annexe_WPT_equivalent}

\section{Proof of Proposition \ref{prop:bounds-mx}}
\label{sec:proof-proposition-bound-X}

\begin{comment}
  The generalization of this result to the homogeneous polynomial basis case is difficult even
  impossible. For the upper bound one will need the Riesz basis property to relate the $L^2$ norm to
  $\ell^2$ norm, which is not the case for the homegeneous polynomials. For the lower bound one will
  need the local property of the basis, which is again not the case for polynomials.
\end{comment}

\begin{proof}
  For any $f\in H^2,g\in H^1$, let $\mbf f, \mbf g$ be the coefficient vectors defined in
  \eqref{eq:Coeff_vec_std}.  By the boundness of $\TauD$, we have
  \begin{align*}
    \abs{\mbf f^\top \mX \mbf g} = \Abs{\TauD(P_L f, P_L g)} \leq C \norm{P_L f}_{H^2} \norm{P_L g}_{H^1}.
  \end{align*}
  Since the scaling function $\phi\in\mcl C_0^r(\R)$ with $r\geq 2$, we have the inverse estimate
  (\cite[Theorem 3.4.1]{cohen_numerical_2003}):
  \begin{align}
    \label{eq:Bernstein341}
    \norm{P_L f}_{H^2} \leq C 2^{-2L}\norm{P_L f}_{L^2} = C 2^{-2L}\norm{\mbf f}_2,
  \end{align}
  where the last identity comes from $\norm{P_L f}_{L^2} = \norm{\mbf f}_{\ell^2}$
  ($\set{\phiLn}_{n\in\Z^2}$ is an orthonormal basis of the approximation space $V_L$). Similarly,
  one has the inverse estimate $\norm{P_L g}_{H^1}\leq C 2^{-L}\norm{\mbf g}_2$. Finally, we obtain
  \begin{align*}
    \abs{\mbf f^\top \mX \mbf g }\leq C 2^{-3L}\norm{\mbf f}_{\ell^2} \norm{\mbf g}_{\ell^2},
  \end{align*}
  which proves the right-hand side of \eqref{eq:bound_WPT_mat_l2}.

  Let $D_\epsilon$ be a circular domain of width $\epsilon$ around $\p D$ defined as
  \begin{align}
    \label{eq:domain_D_e}
    D_{\epsilon} = \set{x\, |\, \text{dist}(x,\p D)\leq \epsilon}.
  \end{align}
  % As $L\rightarrow -\infty$, we choose the constant $\epsilon\propto 2^L$ such that any
  % $\psi_{L,n}^0$ whose support intersecting $\p D$ has its support strictly included in
  % $D_\epsilon$. 
  Let $\eta >\epsilon$ and $D_\eta$ be another circular domain containing $D_\epsilon$, and put
  \begin{align*}
    f(x)=
    \begin{cases}
      x_1 & \text{ if } x\in D_\eta, \\
      0 & \text{ otherwis.}
    \end{cases}
  \end{align*}
  As $L\rightarrow -\infty$ we can choose the constant $\eta \propto 2^L$ such that any $\phiLn$
  whose support intersects $D_\epsilon$ has its support strictly included in $D_\eta$.  Now by
  definition of the wavelet basis $\Xbasis$ in section~\ref{sec:comp-supp-wavel}, the approximation
  space $V_L$ contains polynomials of order $p-1$. Therefore, when restricted on $D_\epsilon$,
  \begin{align}
    \label{eq:f_PL_f_restrict}
    f|_{D_\epsilon}=(P_L f)|_{D_\epsilon}.% = \tilde f|_{D_\epsilon}.  
  \end{align}
  On the other hand, explicit bounds exist on the generalized polarisation tensors
  \eqref{eq:GPT_def} $\mX_{\alpha,\beta}$ (\cite[Lemma 4.12]{ammari_polarization_2007}). In
  particular, for $\abs\alpha=1$ the following estimate holds \cite{capdeboscq}
  \begin{align}
    \label{eq:GPT_bound}
    \abs{D} \leq \frac{\kcst+1}{\abs{\kcst-1}}\Abs{\TauD(x^\alpha,x^\alpha)} \leq C \abs{D}
  \end{align}
  with $C$ being some constant independent of $D$ and $\abs D$ being the volume of $D$. Hence, if $\mbf f$ is
  the coefficient vector of $P_L f$, using \eqref{eq:f_PL_f_restrict} we obtain that
  \begin{align*}
    \abs{\mbf f^\top \mX \mbf f} = \abs{\TauD(P_L f, P_L f)} = \abs{\TauD(f,f)} = \abs{\TauD(x_1,
      x_1)} \geq \frac{\abs{\kcst-1}}{\kcst+1}\abs{D}.
  \end{align*}  
  Finally, notice that $\norm{\mbf f}_{\ell^2}^2 = \norm{P_L f}_{L^2}^2 \leq \norm{f}_{L^2}^2 \lesssim 2^L$,
  which gives the left-hand side of \eqref{eq:bound_WPT_mat_l2}.
\end{proof}

%%% Local Variables: 
%%% mode: latex
%%% TeX-master: "main"
%%% End: 

\bibliographystyle{plain}
\bibliography{Biblio}
\end{document}